\documentclass[a4paper,11pt,leqno]{amsart}
\usepackage{a4wide}
\usepackage{amsmath}
\usepackage[italian, english]{babel}
\usepackage{amsfonts}
\usepackage{amssymb}
\usepackage{dsfont}
\usepackage{mathrsfs}
\usepackage{graphicx}
\usepackage{fancyhdr}
\usepackage{amsthm}
\usepackage{empheq}
\usepackage{cases}
\usepackage[all]{xy}
\usepackage{stmaryrd}
\usepackage[colorlinks, citecolor=blue, linkcolor=red]{hyperref}
\usepackage{mathrsfs}

\usepackage{mathtools}

\def\RR{{\mathbb R}}

\def\gt{\widetilde{g}}

\def\rd{\overset{\circ}{Ric}}
\def\rdc{\overset{\circ}{R}}

\newcommand{\erre}{\mathds{R}}

\newcommand{\cinf}{C^{\infty}(M)}

\newcommand{\ricc}{\operatorname{Ric}}
\newcommand{\diver}{\operatorname{div}}

\newcommand{\ra}{\rightarrow}

\newcommand{\set}[1]{{\left\{#1\right\}}}               
\newcommand{\pa}[1]{{\left(#1\right)}}                  
\newcommand{\sq}[1]{{\left[#1\right]}}                  
\newcommand{\abs}[1]{{\left|#1\right|}}                 
\newcommand{\pair}[1]{\left\langle#1\right\rangle}      

\renewcommand{\hat}[1]{\widehat{#1}}
\renewcommand{\tilde}[1]{\widetilde{#1}}

\newtheorem{ackn}{Acknowledgments\!}

\newtheorem{theorem}{\textbf{Theorem}}[section]
\newtheorem{lemma}[theorem]{\textbf{Lemma}}
\newtheorem{proposition}[theorem]{\textbf{Proposition}}
\newtheorem{cor}[theorem]{\textbf{Corollary}}
\newtheorem{defi}[theorem]{\textbf{Definition}}
\newtheorem{rem}[theorem]{\textbf{Remark}}
\theoremstyle{remark}

\numberwithin{equation}{section}
\title[Four dimensional closed manifolds admit a weak harmonic Weyl metric]
{Four dimensional closed manifolds\\ admit a weak harmonic Weyl metric}

\author[G. Catino]{Giovanni Catino}
\address[Giovanni Catino]{Dipartimento di Matematica, Politecnico di Milano, Piazza Leonardo da Vinci 32, 20133 Milano, Italy}
\email[]{giovanni.catino@polimi.it}

\author[P. Mastrolia]{Paolo Mastrolia}
\address[Paolo Mastrolia]{Dipartimento di Matematica, Universit\`{a} degli Studi di Milano, Via Saldini 50, 20133 Italy.}
\email[]{paolo.mastrolia@unimi.it}

\author[D. D. Monticelli]{Dario D. Monticelli}
\address[Dario Monticelli]{Dipartimento di Matematica, Politecnico di Milano, Piazza Leonardo da Vinci 32, 20133 Milano, Italy}
\email[]{dario.monticelli@polimi.it}

\author[F. Punzo]{Fabio Punzo}
\address[Fabio Punzo]{Dipartimento di Matematica, Politecnico di Milano, Piazza Leonardo da Vinci 32, 20133 Milano, Italy}
\email[]{fabio.punzo@polimi.it}

\begin{document}


\begin{abstract} 
On four-dimensional closed manifolds we introduce a class of canonical Riemannian metrics, that we call {\em weak harmonic Weyl metrics}, defined  as critical points in the conformal class of a quadratic functional involving the norm of the divergence of the Weyl tensor. This class includes Einstein  and, more in general, harmonic Weyl manifolds. We prove that {\em every} closed four-manifold admits a weak harmonic Weyl metric, which is the unique (up to dilations) minimizer of the functional in a suitable conformal class. In general the problem is degenerate elliptic due to possible vanishing of the Weyl tensor. In order to overcome this issue, we minimize the functional in the conformal class determined by a reference metric, constructed by Aubin, with nowhere vanishing Weyl tensor. Moreover, we show that anti-self-dual metrics with positive Yamabe invariant can be characterized by pinching conditions involving suitable quadratic Riemannian functionals. 
\end{abstract}

\maketitle

\begin{center}


\medskip

\centerline{\bf AMS subject classification:  53C20, 53C21, 53C25}

\end{center}

\


\section{Introduction}

Given a closed (i.e., compact without boundary) smooth manifold $M$, it is a natural problem to study canonical Riemannian metrics $g$ on $M$. Many of them can be defined as critical points of certain functionals defined on the space of metrics. Perhaps the most famous one is the {\em Einstein-Hilbert action}
$$
\mathfrak{S}(g):=\operatorname{Vol}_{g}(M)^{-\frac{n-2}{2}}\int_{M} R_{g} \,dV_{g}\,,
$$
where $\operatorname{Vol}_{g}(M)$ and $R_{g}$ denote the volume of $M$ and the scalar curvature of $g$, respectively. All stationary points of $\mathfrak{S}(g)$ are {\em Einstein metrics}, i.e. metrics whose Ricci curvature satisfies $\ricc_{g}=\lambda\, g$, for some $\lambda\in\RR$. While the existence of Einstein metrics as critical points of $\mathfrak{S}(g)$ is not guaranteed (for instance in dimension four due to topological restrictions \cite[Theorem 6.35]{besse}), a constrained version of the problem always admits a solution. More precisely, Yamabe, Aubin, Trudinger, and Schoen (see \cite{leepar}) showed that the {\em Yamabe invariant}
$$
\mathcal{Y}(M,[g]):=\inf_{\gt\in[g]} \mathfrak{S}(\gt)
$$
is always attained in the conformal class $[g]$. Moreover, every critical point in the conformal class of the normalized functional has constant scalar curvature.

In the last decades several curvature conditions generalizing Einstein metrics have been investigated by many authors (see for instance the classical Besse's book \cite{besse} and reference therein). In particular, important examples arise as critical points of functionals which are quadratic in the curvatures (see for instance \cite{cgy, guvia, and12}). In general, the associated Euler-Lagrange equation is of the fourth order in the metric, hence obtaining a satisfactory existence theory can be challenging. 

An important class of metrics which generalizes the Einstein condition is given by {\em harmonic Weyl metrics}, i.e. metrics with divergence-free Weyl tensor, $\delta_{g} W_{g}= 0$ (see again \cite{besse} and \cite{derd}). In fact, it is well known that all Einstein metrics have harmonic Weyl tensor and  that, on four dimensional closed manifolds, there are topological obstructions to the existence of harmonic Weyl metrics (see \cite{bou2tr, dershe}). 

From now on, let $M^4$ be a four-dimensional closed smooth manifold. Observe that all harmonic Weyl metrics are critical points of the quadratic scaling-invariant functional
$$
\mathfrak{D}(g):=\operatorname{Vol}_g(M)^{\frac{1}{2}} \int_M |\delta_{g} W_g|_g^2\, dV_g \,,
$$
while the viceversa in general does not hold. Note that conformal variations give rise to a second order Euler-Lagrange equation, since the transformation law of $\delta W$ (see \cite{besse}) is given by
$$
\delta_{\gt} W_{\gt} = \delta_g W - W_g (\nabla_g u, \cdot, \cdot, \cdot) 
$$ 
for every conformal metric $\gt = e^{2u} g \in [g]$. Thus, in the same spirit of the Yamabe problem, it seems natural to define the conformal invariant
$$
\mathcal{D}(M,[g]):=\inf_{\gt\in[g]}\mathfrak{D}(\gt)\,.
$$
The main purpose of this paper is to study the existence of minimizers in the conformal class for the functional $g\mapsto\mathfrak{D}(g)$. In general the problem is degenerate elliptic due to possible vanishing of the Weyl tensor. In order to overcome this issue, we minimize the functional in the conformal class determined by a reference metric, constructed by Aubin, with nowhere vanishing Weyl tensor. On the contrary, for the Yamabe problem the existence of minimizers is guaranteed in every conformal class.

Besides the aforementioned variational point of view, there is another geometric motivation for studying constrained critical points of $g\mapsto\mathfrak{D}(g)$. Indeed, it was proved by Derdzinski \cite{derd} that, on four manifolds, harmonic Weyl metrics satisfy the nice Weitzenb\"ock formula
\begin{equation}\label{weit}
\frac{1}{2}\Delta |W|^{2} = |\nabla W|^{2}+\frac{1}{2}R|W|^{2}-3\,W_{ijkl}W_{ijpq}W_{klpq} 
\end{equation}
(see the next section for the precise notation). On the other hand, Chang, Gursky and Yang \cite{cgy2} showed that, on {\em every} closed four-manifold $(M^{4},g)$, the following integral formula holds (see Corollary \ref{l-intid}):
\begin{equation}\label{intCGY}
\int_{M}\left(|\nabla W|^{2}-4|\delta W|^{2}+\frac{1}{2}R|W|^{2}-3\,W_{ijkl}W_{ijpq}W_{klpq} \right)\,dV = 0\,.
\end{equation}
A simple consequence is that, on a closed four-manifold $M^{4}$,
$$
\delta_g W_g = 0 \quad\quad\Longleftrightarrow \quad\quad \textrm{Equation \eqref{weit} holds on } (M^4,g)\,.
$$
In Section \ref{ssec-cm} we show that a metric is critical in the conformal class for the functional $g\mapsto\mathfrak{D}(g)$ if and only if it satisfies  the Weitzenb\"ock formula
\begin{equation}\label{eq-WHW}
\frac{1}{2}\Delta |W|^{2} = |\nabla W|^{2}+\frac{1}{2}R|W|^{2}-3\,W_{ijkl}W_{ijpq}W_{klpq}-8|\delta W|^{2}+\frac{4}{\operatorname{Vol}(M)}\int_M|\delta W|^{2}\,dV \,.
\end{equation}
Note that this equation reduces to \eqref{weit} if $\delta W=0$. Hence we are led to give the following

\begin{defi}\label{def_WHW} Let $M^{4}$ be a closed four-dimensional manifold. A Riemannian metric $g$ on $M^{4}$ is a {\em weak harmonic Weyl metric} if the Weitzenb\"ock equation \eqref{eq-WHW} holds on $(M^4,g)$.
\end{defi}
\noindent Clearly, harmonic Weyl metrics (and Einstein metrics) are weak harmonic Weyl metrics. We explicitly observe that integrating equation \eqref{eq-WHW} we obtain the identity  \eqref{intCGY} and this gives no {\em a priori} obstructions to the existence of weak harmonic Weyl metrics, contrary to what happens with  \eqref{weit}.

\smallskip

Our first main result is the following

\begin{theorem}\label{t-main} On every closed four-dimensional manifold there exists a weak harmonic Weyl metric.
\end{theorem}

\noindent {\em Remarks:}

\begin{itemize} 

\item[1.]  Aubin \cite{aubin} proved that every closed Riemannian manifold admits a constant negative scalar curvature metric. Besides this one, to the best of our knowledge,  Theorem \ref{t-main} is the only existence result of a  {\em canonical} metric, which generalizes the Einstein condition, on {\em every} four-dimensional Riemannian manifold, without any topological obstructions. 
 
\item[2.] To be more precise, the metric in Theorem \ref{t-main} is constructed as follows: first, thanks to a result of Aubin \cite[Section 4]{aubin}, on every four-dimensional manifold $M^4$ we can choose a reference metric $g_0$ with $|W_{g_0}|_{g_0}>0$. Then, we prove that on $(M^4, g_0)$  the infimum $\mathcal{D}(M,[g_0])$ is attained by a conformal metric $g\in[g_0]$, which is a weak harmonic Weyl metric. Moreover, we show that every critical point in the conformal class $[g_0]$ is necessarily a minimum point.  

\item[3.] From both the geometric and the analytic point of view, it would be interesting to understand which conformal classes of metrics contain a weak harmonic Weyl representative.

\item[4.] We can also consider the (anti-)self-dual functional
$$
\mathfrak{D}^{\pm}(g):=\operatorname{Vol}_g(M)^{\frac{1}{2}} \int_M |\delta_{g} W^{\pm}_g|_g^2\, dV_g \,,
$$
and define {\em weak half harmonic Weyl metrics} its critical points in the conformal class. In particular we can prove that, given a closed four-manifold $(M^4, g_0)$ with $|W^{\pm}_{g_0}|_{g_0}>0$, there exists a weak half harmonic Weyl metric $\gt\in[g_{0}]$. However, we do not know if the aforementioned result by Aubin can be extended to the (anti-)self-dual Weyl tensor $W^{\pm}$.
\end{itemize}

\smallskip

In order to prove this theorem, we endow a closed four-manifolds $M^4$ with the metric $g_0$ constructed by Aubin and we consider the functional
$$
\mathfrak{D}(v):=\mathfrak{D}(v^{-2} g_0)=  \left(\int_M v^{-4} dV\right)^\frac{1}{2} \int_M \pa{\frac{1}{4}|W|^{2}|\nabla v|^{2}+\abs{\delta W}^2 v^{2} -  (v)^{2}_{s}\, W_{sijk}W_{pijk, p} }\,dV \,,
$$
where all the geometric quantities are referred to $g_0$ and the function $v$ belongs to the convex cone
$$
H(M):=\left\{u\in H^1(M): u>0 \,\,\text{a.e. and}\,\int_M u^{-4}\,dV<\infty \right\}\,.
$$ 
The condition $|W|>0$ is crucial, as it implies the uniform ellipticity of the problem. A variational argument, combined with some spectral analysis and maximum principles, shows that $u\mapsto\mathfrak{D}(u)$ admits a minimum point $v$ in $H(M)$. Consequently, $v$ is a (weak) solution of the Euler-Lagrange equation
$$
-\frac{1}{4}\diver(|W|^2 \nabla v) + \left(|\delta W|^2+\diver(W_{sijk}W_{pijk,p})\right)v = \mathfrak{D}(v)\left(\int_M v^{-4} dV\right)^{-3/2}\, \frac{1}{v^5} \,,
$$
which is a uniformly elliptic semilinear equation with singular nonlinearity. Here, again, all the geometric quantities are referred to $g_0$. Hence, by standard elliptic regularity theory, $v\in C^\infty (M)$ and 
$$
\mathcal{D}(M,[g_0]) = \min_{0<u\in C^{\infty}(M)} \mathfrak{D}(u) = \min_{u\in H(M)} \mathfrak{D}(u)\,.
$$
Therefore 
$$
g:= v^{-2} g_0
$$
is a weak harmonic Weyl metric on $M^4$.

\

In the second part of the paper, we prove a characterization of {\em anti-self-dual} four-dimensional manifolds, i.e. $(M^{4},g$) with $W_g^{+}\equiv 0$, assuming the positivity of the Yamabe invariant, a pinching condition on the conformal invariant
$$
\mathcal{W^{+}}(M, [g]):=\int_{M}|W^{+}_{g}|_{g}^{2}\,dV_{g}
$$
and the non-positivity of the modified functional
$$
\mathfrak{D}_{\alpha}^{+}(g) := \operatorname{Vol}_g(M)^{\frac12}\left( \int_M |\delta_{g} W_g^+|_g^2\, dV_g -\frac{5-9\alpha}{24}\int_{M}R_{g}|W^{+}_{g}|^{2}_{g}\,dV_{g}\right)\,,
$$
defined for a given $\alpha\in[0, \frac{5}{9}]$. In the same spirit we define $\mathfrak{D}_{\alpha}^{-}(g)$. Note that $\mathfrak{D}_{\frac59}^{\pm}(g)=\mathfrak{D}^{\pm}(g)$. More precisely, we have the following

\begin{theorem}\label{t-rigidity} Let $(M^{4},g)$ be a closed Riemannian manifold with positive Yamabe invariant $\mathcal{Y}(M,[g])>0$. Then $(M^{4}, g)$ is anti-self-dual, i.e. $W^{+}\equiv 0$, if and only if, there exists $\alpha\in[0, \frac{5}{9}]$ such that
\begin{equation*}
\mathcal{W^{+}}(M, [g]) \leq \frac{\alpha^{2}}{6} \mathcal{Y}(M,[g])^{2} \quad\quad\text{and}\quad\quad
\mathfrak{D}_{\alpha}^{+}(\gt)\leq 0 \quad\textrm{for some }\,\gt\in[g] \,.
\end{equation*}
\end{theorem}
\noindent The same result holds for the anti-self-dual part $W^{-}$ of the Weyl tensor. As a consequence we can prove the following lower bound for $\mathcal{W^{+}}(M, [g])$:

\begin{cor}\label{c-asf}  Let $(M^{4},g)$ be a closed Riemannian manifold with positive Yamabe invariant $\mathcal{Y}(M,[g])>0$. Suppose that there exists $\alpha\in[0, \frac{5}{9}]$ such that 
$$
\mathfrak{D}_{\alpha}^{\pm}(\gt)\leq 0 \quad\textrm{for some }\,\gt\in[g]\,.
$$
Then either $W_g^{\pm}\equiv 0$ or
$$
\mathcal{W}^{\pm}(M,[g]) > \left(\frac{3\alpha^{2}}{1+2\alpha^{2}}\right) \frac{16\pi^{2}}{3} \left(2 \chi(M)\pm 3 \tau(M)\right)\,,
$$
where $\chi(M)$ and $\tau(M)$ denote the Euler characteristic and the signature of $M$, respectively.
\end{cor}

\noindent Remarks:

\begin{itemize}

\item[1.] Gursky \cite{gur2} proved that, if $\delta W^{\pm}\equiv 0$, i.e. $\mathfrak{D}^{\pm}_{\frac{5}{9}}(g)\leq 0$, on a four manifolds $(M^{4},g)$ with positive Yamabe invariant, then either $W_g^{\pm}\equiv 0$ or
$$
\mathcal{W}^{\pm}(M,[g]) \geq \frac{16\pi^{2}}{3} \left(2 \chi(M)\pm 3 \tau(M)\right)\,.
$$
The conclusion in this case is stronger than the one in Corollary \ref{c-asf}. Ideed, the harmonic Weyl condition implies the validity of the pointwise Weitzenb\"ock formula \eqref{weit} which allows to conclude by using a clever Yamabe-type argument.

\item[2.] The same estimate as in Corollary \ref{c-asf} with $\alpha=\frac{1}{3}$ appeared in \cite[Theorem 4.1]{caotra}. The authors proved the lower bound on $\mathcal{W}^{\pm}(M,[g])$, assuming $(M^{4}, g)$ being a {\em gradient shrinking Ricci soliton} satisfying
$$
\int_M |\delta W^{\pm}|^2\,dV \leq \frac{1}{12}\int_M R|W^{\pm}|^2\,dV \,,
$$ 
i.e. $\mathfrak{D}^{\pm}_{\frac{1}{3}}(g)\leq 0$. Note that in Corollary \ref{c-asf} we do not assume any curvature condition on the Ricci tensor.

\item[3.] We conjecture that Theorem \ref{t-rigidity} and Corollary \ref{c-asf} hold only assuming one of the conformal invariant conditions
$$
\mathcal{D}_{\alpha}^{\pm}(M,[g]):=\inf_{\gt\in[g]}\mathfrak{D}_{\alpha}^{\pm}(\gt) \leq 0 \,.
$$
The conjecture would follow, if one can show that the infimum $\mathcal{D}_{\alpha}^{\pm}(M,[g])$ is achieved. This could be obtained by arguments similar to those used in the proof of Theorem \ref{t-main}, further assuming $|W^\pm_{\gt}|_{\gt}>0$ for some $\gt\in[g]$.  Without such extra condition, the functional and the associated Euler-Lagrange equation are {\em degenerate}, thus a completely different  analysis has to be performed.
\end{itemize}

\smallskip

The paper is organized as follows. Section \ref{sec-pre} is devoted to the geometric preliminaries: we give the relevant definitions, we prove a general Weitzenb\"ock formula for the Weyl tensor and we recall (and define) some Riemannian functionals. In Section \ref{sec-eul} we derive the Euler-Lagrange equation satisfied by critical points in the conformal class of the functional $g\mapsto\mathfrak{D}(g)$. The existence of solutions to the elliptic equation and the proof of Theorem \ref{t-main} are given in Section \ref{sec-ex}. Section \ref{sec-kato} is devoted to the proof of a new quantitative Kato inequality for the Weyl tensors $W^\pm$ which is used, in Section \ref{sec-rigidity}, to prove the rigidity results Theorem \ref{t-rigidity} and Corollary \ref{c-asf}.

\

\section{Preliminaries}\label{sec-pre}

The Riemann curvature
operator of an oriented Riemannian manifold $(M^n,g)$ is defined by
$$
\mathrm{R}(X,Y)Z=\nabla_{X}\nabla_{Y}Z-\nabla_{Y}\nabla_{X}Z-\nabla_{[X,Y]}Z\,.
$$
Throughout the article, the Einstein convention of summing over the repeated indices will be adopted. In a local coordinate system the components of the $(1, 3)$-Riemann
curvature tensor are given by
$R^{l}_{ijk}\tfrac{\partial}{\partial
  x^{l}}=\mathrm{R}\big(\tfrac{\partial}{\partial
  x^{j}},\tfrac{\partial}{\partial
  x^{k}}\big)\tfrac{\partial}{\partial x^{i}}$ and we denote  by $Riem$ its $(0,4)$ version with components by
$R_{ijkl}=g_{im}R^{m}_{jkl}$. The Ricci tensor is obtained by the contraction
$R_{ik}=g^{jl}R_{ijkl}$ and $R=g^{ik}R_{ik}$ will
denote the scalar curvature ($g^{ij}$ are the coefficient of the inverse of the metric $g$). Moreover, we will denote by $(\rd)_{ik}=\rdc_{ik}=R_{ik}-\frac{1}{n}R\, g_{ik}$ the traceless Ricci tensor. The so called Weyl tensor is then
defined by the following decomposition formula in dimension $n\geq 3$,
\begin{eqnarray}
\label{Weyl}
W_{ijkl}  & = & R_{ijkl} \, - \, \frac{1}{n-2} \, (R_{ik}g_{jl}-R_{il}g_{jk}
+R_{jl}g_{ik}-R_{jk}g_{il})  \nonumber \\
&&\,+\frac{R}{(n-1)(n-2)} \,
(g_{ik}g_{jl}-g_{il}g_{jk})\, \, .
\end{eqnarray}
The Weyl tensor shares the symmetries of the curvature
tensor. Moreover, as it can be easily seen by the formula above, all of its contractions with the metric are zero, i.e. $W$ is totally trace-free. In dimension three, $W$ is identically zero on every Riemannian manifold, whereas, when $n\geq 4$, the vanishing of the Weyl tensor is
a relevant condition, since it is  equivalent to the local
  conformal flatness of $(M^n,g)$. We also recall that in dimension $n=3$,  local conformal
  flatness is equivalent to the vanishing of the Cotton tensor
\begin{equation}\label{def_cot}
C_{ijk} =  R_{ij,k} - R_{ik,j}  -
\frac{1}{2(n-1)}  \big( R_k  g_{ij} -  R_j
g_{ik} \big)\,,
\end{equation}
where $R_{ij,k}=\nabla_k R_{ij}$ and $R_k=\nabla_k R$ denote, respectively, the components of the covariant derivative of the Ricci tensor and of the differential of the scalar curvature.
By direct computation, we can see that the Cotton tensor $C$
satisfies the following symmetries
\begin{equation}\label{CottonSym}
C_{ijk}=-C_{ikj},\,\quad\quad C_{ijk}+C_{jki}+C_{kij}=0\,,
\end{equation}
moreover it is totally trace-free,
\begin{equation}\label{CottonTraces}
g^{ij}C_{ijk}=g^{ik}C_{ijk}=g^{jk}C_{ijk}=0\,,
\end{equation}
by its skew--symmetry and Schur lemma. We recall that, for $n\geq 4$,  the Cotton tensor can also be defined as one of the possible divergences of the Weyl tensor:
\begin{equation}\label{def_Cotton_comp_Weyl}
C_{ijk}=\pa{\frac{n-2}{n-3}}W_{tikj, t}=-\pa{\frac{n-2}{n-3}}W_{tijk, t}=-\frac{n-2}{n-3} (\delta W)_{ijk} \,.
\end{equation}
A computation shows that the two definitions coincide (see e.g. \cite{besse}).

We say that a $n$-dimensional, $n\geq 3$, Riemannian manifold $(M^{n},g)$ is an {\em Einstein manifold} if the Ricci tensor satisfies
$$
\ricc \,=\, \lambda g \,,
$$
for some $\lambda\in\erre$. In particular $R=n\lambda \in\erre$ and the Cotton tensor $C$ vanishes. If $n\geq 4$, equation \eqref{def_Cotton_comp_Weyl} implies that the divergence of the Weyl tensor and thus of the Riemann tensor are identically null, i.e.
\begin{equation}\label{harmall}
W_{tijk,t} \,=\, 0\,, \quad R_{tijk,t}=0
\end{equation}
on every Einstein manifold. Manifolds satisfying these curvature conditions are said to have {\em harmonic Weyl curvature} or {\em harmonic curvature}, respectively. The Hessian $\nabla^{2}$ of some tensor $T$ of local components $T_{i_{1}\dots i_{k}}^{j_{1}\dots j_{l}}$ will be
$$
(\nabla^{2}T)_{pq}=\nabla_{q}\nabla_{p}T_{i_{1}\dots i_{k}}^{j_{1}\dots j_{l}}=T_{i_{1}\dots i_{k},pq}^{j_{1}\dots j_{l}}
$$
and similarly $\nabla^{k}$ for higher derivatives. The (rough) Laplacian of a tensor $T$ is given by $\Delta T_{i_{1}\dots i_{k}}^{j_{1}\dots j_{l}} =g^{pq}T_{i_{1}\dots i_{k},pq}^{j_{1}\dots j_{l}}$. The Riemannian metric induces norms on all the
 tensor bundles, and in coordinates the squared norm is given by
$$
|T|^{2}=g^{i_{1}m_{1}}\cdots g^{i_{k}m_{k}}
 g_{j_{1}n_{1}}\dots g_{j_{l}n_{l}} T_{i_{1}\dots
   i_{k}}^{j_{1}\dots j_{l}} T_{m_{1}\dots m_{k}}^{n_{1}\dots
   n_{l}}\,.
$$

\smallskip

\subsection{Dimension four} In this subsection we recall  some known identities involving the Weyl tensor on four dimensional Riemannian manifold. First we recall that, if $T=\{T_{ijkl}\}$ is a tensor with the same symmetries of the Riemann tensor  (algebraic curvature tensor), it defines a symmetric operator, $T:\Lambda^{2}\longrightarrow \Lambda^{2}$ on the bundle of two-forms $\Lambda^{2}$ by
\begin{equation}\label{conv}
(T \omega)_{kl} \,:=\,\frac{1}{2}T_{ijkl}\omega_{ij}\,,
\end{equation}
with $\omega\in\Lambda^{2}$. Hence we have that $\lambda$ is an eigenvalue of $T$ if $T_{ijkl}\omega_{ij} = 2\lambda\,\omega_{kl}$, for some $0\neq \omega\in\Lambda^{2}$; note that the operator norm on $\Lambda^{2}$ satisfies $\Vert T\Vert^{2}_{\Lambda^{2}}=\frac{1}{4}|T|^{2}$.

The key feature is that $\Lambda^{2}$, on an oriented Riemannian manifold of dimension four $(M^{4},g)$, decomposes as the sum of two subbundles $\Lambda^{\pm}$, i.e.
\begin{equation}\label{dec}
\Lambda^{2}=\Lambda^{+} \oplus \Lambda^{-}\,.
\end{equation}
These subbundles are by definition the eigenspaces of the Hodge operator
$$
\star:\Lambda^{2}\rightarrow \Lambda^{2}
$$
corresponding respectively to the eigenvalue $\pm 1$. In the literature, sections of $\Lambda^{+}$ are called {\em self-dual} two-forms, whereas sections of $\Lambda^{-}$ are called {\em anti-self-dual} two-forms. Now, since the curvature tensor $Riem$ may be viewed as a map $\mathcal{R}:\Lambda^{2}\ra\Lambda^{2}$, according to \eqref{dec} we have the curvature decomposition
\begin{displaymath}
\mathcal{R}=\left(\begin{array}{c|c}
W^{+}+\frac{R}{12}\,I & \overset{\circ}{Ric} \\
\hline
\overset{\circ}{Ric} & W^{-}+\frac{R}{12}\,I \end{array}\right), \end{displaymath}
where
$$
W = W^{+} + W^{-}
$$
and the self-dual and anti-self-dual $W^{\pm}$ are trace-free endomorphisms of $\Lambda^{\pm}$, $I$ is the identity map of $\Lambda^{2}$ and $\overset{\circ}{Ric}$ represents the trace-free Ricci curvature $Ric-\frac{R}{4}g$.

Following Derdzinski \cite{derd}, for $x\in M^{4}$, we can choose an oriented orthogonal basis $\omega^{+}, \eta^{+}, \theta^{+}$ (respectively, $\omega^{-}, \eta^{-}, \theta^{-}$) of $\Lambda^{+}_{x}$ (respectively, $\Lambda^{-}_{x}$), consisting of eigenvectors of $W^{\pm}$ such that $|\omega^{\pm}|=|\eta^{\pm}|=|\theta^{\pm}|=\sqrt{2}$ and, at $x$,
\begin{equation}\label{eq-derw}
W^{\pm} \,=\, \frac{1}{2}\big(\lambda^{\pm}\omega^{\pm}\otimes\omega^{\pm}+\mu^{\pm} \eta^{\pm}\otimes\eta^{\pm}+\nu^{\pm}\theta^{\pm}\otimes\theta^{\pm}\big)
\end{equation}
where $\lambda^{\pm}\leq\mu^{\pm}\leq\nu^{\pm}$ are the eigenvalues of $W^{\pm}_{x}$. Since $W^{\pm}$ are trace-free, one has $\lambda^{\pm}+\mu^{\pm}+\nu^{\pm}=0$. By definition, we have
$$
\Vert W^{\pm}\Vert^{2}_{\Lambda^{2}}=(\lambda^{\pm})^{2}+(\mu^{\pm})^{2}+(\nu^{\pm})^{2}=\frac{1}{4}|W^\pm|^2\,.
$$
We recall that the orthogonal basis $\omega^{\pm}, \eta^{\pm}, \theta^{\pm}$ forms a quaternionic structure on $T_{x}M$ (see \cite[Lemma 2]{derd}), namely in some local frame
$$
\omega^{\pm}_{ip}\omega^{\pm}_{pj}\,=\,\eta^{\pm}_{ip}\eta^{\pm}_{pj}\,=\,\theta^{\pm}_{ip}\theta^{\pm}_{pj}\,=\,-\delta_{ij}\,,
$$
$$
\omega^{\pm}_{ip}\eta^{\pm}_{pj}=\theta^{\pm}_{ij},\quad\quad \eta^{\pm}_{ip}\theta^{\pm}_{pj}=\omega^{\pm}_{ij},\quad\quad \theta^{\pm}_{ip}\omega^{\pm}_{pj}=\eta^{\pm}_{ij}\,.
$$

The following identity on the Weyl tensor in dimension four is known (see \cite{derd})

\begin{equation}\label{eq-ww}
W^\pm_{ijkt}W^\pm_{ijkl} = \frac{1}{4}|W^\pm|^{2}g_{tl} = \Vert W^\pm\Vert^{2}_{\Lambda^{2}}g_{tl}\,.
\end{equation}

As far as the covariant derivative of Weyl is concerned, it can be shown that (see again \cite{derd}), locally, one has
\begin{align}\label{eq-derder}
2 \nabla W^{\pm} &= \big(d\lambda^{\pm}\otimes\omega^{\pm}+(\lambda^{\pm}-\mu^{\pm})c^{\pm}\otimes\eta^{\pm}+(\nu^{\pm}-\lambda^{\pm})b^{\pm}\otimes\theta^{\pm}\big)\otimes\omega^{\pm}\\\nonumber
&+\big(d\mu^{\pm}\otimes\eta^{\pm}+(\lambda^{\pm}-\mu^{\pm})c^{\pm}\otimes\omega^{\pm}+(\mu^{\pm}-\nu^{\pm})a^{\pm}\otimes\theta^{\pm}\big)\otimes\eta^{\pm}\\\nonumber
&+\big(d\nu^{\pm}\otimes\theta^{\pm}+(\nu^{\pm}-\lambda^{\pm})b^{\pm}\otimes\omega^{\pm}+(\mu^{\pm}-\nu^{\pm})a^{\pm}\otimes\eta^{\pm}\big)\otimes\theta^{\pm}\,,
\end{align}
for some one forms $a^{\pm},b^{\pm},c^{\pm}$.
By orthogonality,  we get
$$
\Vert\nabla W\Vert^2_{\Lambda^{2}} = \Vert\nabla W^+\Vert^2_{\Lambda^{2}} +  \Vert\nabla W^-\Vert^2_{\Lambda^{2}}
$$
and
\begin{equation}\label{eq-nqder}
\Vert\nabla W^{\pm}\Vert^{2}_{\Lambda^{2}} = |d\lambda^{\pm}|^{2}+|d\mu^{\pm}|^{2}+|d\nu^{\pm}|^{2}  +2(\mu^{\pm}-\nu^{\pm})^{2}|a^{\pm}|^{2}+2(\lambda^{\pm}-\nu^{\pm})^{2}|b^{\pm}|^{2}+ 2(\lambda^{\pm}-\mu^{\pm})^{2}|c^{\pm}|^{2} \,,
\end{equation}
where $\Vert\nabla W^{\pm}\Vert^{2}_{\Lambda^{2}}=\frac{1}{4}|\nabla W^\pm|^2$. 

Finally, the following identity holds (see \cite{catmas})

\begin{lemma}\label{l-urka}
   On every $n$-dimensional, $n\geq 4$, Riemannian manifold one has
  \begin{equation*}
    W_{ijkl, t}W_{ijkt, l} = \frac{1}{2}\abs{\nabla W}^2 - \frac{1}{n-3}\abs{\delta W}^2.
  \end{equation*}
  In particular, on a four manifold  one has
   \begin{equation}\label{GradWeylNormEinstein}
   W^\pm_{ijkl, t}W^\pm_{ijkt, l} = \frac{1}{2}\abs{\nabla W^\pm}^2 - \abs{\delta W^\pm}^2.
  \end{equation}
\end{lemma}

\smallskip

\subsection{A general Weitzenb\"ock formula for the Weyl tensor} In this subsection, we prove that on every $n$-dimensional Riemannian manifold, $n\geq 4$, the Weyl tensor satisfies a nice Weitzenb\"ock formula. Namely we have

\begin{lemma}\label{l-intidGenDim} Let $(M^{n},g)$, $n\geq 4$, be a $n$-dimensional  Riemannian manifold. Then
\[
\frac{1}{2}\Delta|W|^{2}=\abs{\nabla W}^2-2\pa{\frac{n-2}{n-3}}\abs{\delta W}^2 +2R_{pq}W_{pikl}W_{qikl}-3W_{ijkl}W_{ijpq}W_{klpq}-2\pa{W_{ijkl}C_{jkl}}_i  \,.
\]
\end{lemma}
\begin{proof}
  From the second Bianchi identity for the Weyl tensor (see for instance \cite{cgy2, catmasmonrig}) we have
  \[
  -W_{klij, mm}-W_{klmi, jm}+W_{mjkl, im} = \Psi_{ijkl},
  \]
  where
  \[
  \Psi_{ijkl}:=\frac{1}{n-2}\pa{C_{ljm, m}\delta_{ki}+C_{lmi, m}\delta_{kj}+C_{lij, k}-C_{kjm, m}\delta_{li}-C_{kmi, m}\delta_{lj}-C_{kij, l}}.
  \]
  The previous relation can be rewritten as
  \begin{align*}
    W_{klij, mm} &= \pa{W_{mjkl, mi}-W_{mikl, mj}} - \Psi_{ijkl} +\pa{W_{klmj, im}-W_{klmj, mi}} - \pa{W_{klmi, jm}-W_{klmi, mj}}\\ &=\pa{\frac{n-3}{n-2}}\pa{C_{ikl, j}-C_{jkl, i}} - \Psi_{ijkl} +\pa{W_{klmj, im}-W_{klmj, mi}} - \pa{W_{klmi, jm}-W_{klmi, mj}}.
  \end{align*}
  Using the commutation relation for the second covariant derivative of the Weyl tensor (see \cite{catmasmonrig}) to expand the two terms $W_{klmj, im}-W_{klmj, mi}$ and $W_{klmi, jm}-W_{klmi, mj}$, and also the first Bianchi identity for $W$, we deduce
  \begin{align*}
  \Delta W_{ijkl}&=\pa{\frac{n-3}{n-2}}\pa{C_{ikl, j}-C_{jkl, i}} - \Psi_{ijkl} \\&+R_{ip}W_{pjkl}-R_{jp}W_{pikl}-2\pa{W_{ipjq}W_{pqkl}-W_{ipql}W_{jpqk}+W_{ipqk}W_{jpql}}\\&+\frac{1}{n-2}\sq{R_{jp}W_{pikl}-R_{ip}W_{pjkl}-R_{lp}\pa{W_{pikj}-W_{pjki}}-R_{kp}\pa{W_{pjli}-W_{pilj}}} \\&+\frac{1}{n-2}R_{pq}\pa{W_{piql}\delta_{kj}-W_{pjql}\delta_{ki}+W_{pikq}\delta_{lj}-W_{pjkq}\delta_{li}}.
  \end{align*}
  Contracting with $W_{ijkl}$ and exploiting again the first Bianchi identity, we obtain formula \eqref{l-intidGenDim}.
\end{proof}

In dimension four, using identity \eqref{eq-ww} and the orthogonality of $W^\pm$, the formula simplifies to the following 

\begin{cor}\label{l-intid} Let $(M^{4},g)$ be a four dimensional  Riemannian manifold. Then
$$
\frac{1}{2}\Delta|W|^{2}=\abs{\nabla W}^2 -4\abs{\delta W}^2 +\frac{1}{2}R\abs{W}^2-3W_{ijkl}W_{ijpq}W_{klpq}-2\pa{W_{ijkl}C_{jkl}}_i \,.
$$
As a consequence, if $M$ is closed one has the integral identity (see \cite{cgy2}):
$$
\int_{M}\Big(|\nabla W|^{2}-4|\delta W|^{2}+\frac{1}{2}R|W|^{2}-3 W_{ijkl}W_{ijpq}W_{klpq}\Big)\,dV = 0 \,.
$$
Moreover, we have
$$
\int_{M}\Big(|\nabla W^{\pm}|^{2}-4|\delta W^{\pm}|^{2}+\frac{1}{2}R|W^{\pm}|^{2}-3 W^{\pm}_{ijkl}W^{\pm}_{ijpq}W^{\pm}_{klpq}\Big)\,dV = 0 \,.
$$
\end{cor}

\smallskip

\subsection{Some Riemannian functionals}

Let $(M^{4},g)$ be a closed four-dimensional Riemannian manifold. First of all we recall the Chern-Gauss-Bonnet formula and the Hirzebruch signature formula (see~\cite[Equation 6.31]{besse})
\begin{equation}\label{cgb}
\int_{M}\left(|W^{+}|^{2}+|W^{-}|^{2}-2|\rd|^{2}+\frac{1}{6}R^{2} \right)\,dV = 32\pi^{2}\chi(M) \,,
\end{equation}
\begin{equation}\label{hs}
\int_{M}\left(|W^{+}|^{2} - |W^{-}|^{2}\right)\,dV = 48 \pi^{2} \tau(M)\,.
\end{equation}
If we denote with $\sigma_{2}(A)$ the second-elementary function of the eigenvalues of the Schouten tensor $A:=\frac{1}{2}\left(Ric-\frac{1}{6}R\,g\right)$, it is easy to see that 
$$
\sigma_{2}(A) \,=\, \frac{1}{96}R^{2} - \frac{1}{8}|\rd|^{2}
$$
and the Chern-Gauss-Bonnet formula reads
$$
\int_{M}\left(|W^{+}|^{2}+|W^{-}|^{2}+16\sigma_{2}(A)\right)\,dV = 32\pi^{2}\chi(M) \,.
$$
In particular, using \eqref{hs}, we get
\begin{equation}\label{comb}
8\int_{M}\sigma_{2}(A)\,dV = \int_{M}|W^{\pm}|^{2}\,dV+8\pi^{2}\left(2\chi(M)\pm\tau(M)\right) \,.
\end{equation}
Observing that the $L^{2}$-norm of the Weyl tensors $W^\pm$ in dimension four are conformally invariant and we define 
$$
\mathcal{W^{\pm}}(M, [g]):=\int_{M}|W^{\pm}_{g}|_{g}^{2}\,dV_{g}\,.
$$
In particular, it follows that the integral of $\sigma_{2}(A)$ is conformally invariant too. We denote $\mathcal{Y}(M,[g])$ the Yamabe invariant associated to $(M^{4},g)$ (here $[g]$ is the conformal class of $g$) defined by
$$
\mathcal{Y}(M,[g]) := \inf_{\gt\in[g]} \frac{\int_{M}\widetilde{R}\,dV_{\gt}}{\operatorname{Vol}_{\gt}(M)^{\frac12}} = 6\inf_{u\in W^{1,2}(M)} \frac{\int_{M}|\nabla u|^{2}\,dV+\frac{1}{6}\int_{M}R\,u^{2}\,dV}{\left(\int_{M} u^{4}\,dV\right)^{\frac12}}
$$
It is well known that, on a closed manifold, $\mathcal{Y}(M,[g])$ is positive (respectively zero or negative) if and only if there exists a conformal metric in the conformal class $[g]$ with everywhere positive (respectively zero or negative) scalar curvature. We recall the following lower bound for the Yamabe invariant which was proved by Gursky~\cite{gur127}. 
\begin{lemma}\label{lem-gur} Let $(M^{4},g)$ be a closed four-dimensional manifold. Then, the following estimate holds
$$
\mathcal{Y}(M,[g])^{2} \geq 96\int_{M}\sigma_{2}(A)\,dV_{g} \,=\,\int_{M}R^{2}\,dV_{g} - 12\int_{M}|\rd|^{2}\,dV_{g}\,.
$$
Equivalently,
$$
\mathcal{Y}(M,[g])^{2} \geq 12 \,\mathcal{W}^\pm(M,[g])+96\pi^{2}\left(2\chi(M)\pm\tau(M)\right)\,.
$$
Moreover, the inequality is strict unless $(M^{4},g)$ is conformal to an Einstein manifold. 
\end{lemma}

As anticipated in the Introduction, we now define the quadratic, scale-invariant functionals given by
$$
\mathfrak{D}^\pm(g) = \operatorname{Vol}_g(M)^{\frac12} \int_M |\delta_{g} W_g^\pm|_g^2\, dV_g \,.
$$
Let also
$$
\mathfrak{D}(g) = \mathfrak{D}^{+}(g)+\mathfrak{D}^{-}(g)=\operatorname{Vol}_g(M)^{\frac12} \int_M |\delta_{g} W_g|_g^2\, dV_g \,.
$$
We also denote with
$$
\mathcal{D}^{\pm}(M,[g])=\inf_{\gt\in[g]}\mathfrak{D}^{\pm}(\gt) \,
$$
and
$$
\mathcal{D}(M,[g])=\inf_{\gt\in[g]}\mathfrak{D}(\gt) \,.
$$
It is clear that if $[g]$ contains a metric with $\delta W^{\pm}\equiv 0$, then $\mathcal{D}^{\pm}(M,[g])=0$. For example, if $(M^{4},g)$ is K\"ahler with positive scalar curvature $R_g$, then $\gt=R_{g}^{-2}g$ satisfies $\delta_{\gt} W_{\gt}^+ \equiv 0$ and thus $\mathcal{D}^{+}(M,[g])=0$.

Finally, for $\alpha\in[0,\frac59]$ we define the Riemannian functionals
$$
\mathfrak{D}_{\alpha}^{\pm}(g) = \operatorname{Vol}_g(M)^{\frac12}\left( \int_M |\delta_{g} W_g^\pm|_g^2\, dV_g -\frac{5-9\alpha}{24}\int_{M}R_{g}|W^{\pm}_{g}|^{2}_{g}\,dV_{g}\right)\,,
$$
and its infimum
$$
\mathcal{D}_{\alpha}^{\pm}(M,[g])=\inf_{\gt\in[g]}\mathfrak{D}_{\alpha}^{\pm}(\gt) \,.
$$
Note that
$$
\mathfrak{D}_{\frac59}^{\pm}(g)=\mathfrak{D}^{\pm}(g)\,.
$$

\

\section{The Euler-Lagrange equation} \label{sec-eul}

Let $M^4$ be a closed smooth manifold. In this section we derive the Euler-Lagrange equations satisfied, respectively, by a critical metric in the conformal class of the functional 
$$
g\mapsto \mathfrak{D}(g) =\operatorname{Vol}_g(M)^{\frac{1}{2}} \int_M |\delta_{g} W_g|_g^2\, dV_g \,,
$$
and by the conformal factor.

\subsection{Critical metrics}\label{ssec-cm}
Let $U\in\cinf$ and $u(x, t) : M\times \mathds{R} \ra M$ be a  $1$-parameter family of smooth functions such that $u(x, 0)\equiv 0$, $\frac{d u(x, t)}{dt}\bigg|_{t=0}=U(x)$; for a conformal change of the metric of the form
\begin{equation}
  \tilde{g}(t) = e^{2u(x, t)}g,
\end{equation}
from the formula for the conformal change of $\delta W$ (see \cite{besse}) we easily deduce that
\begin{equation}
  e^{3u(x, t)}\widetilde{W}_{ijkl, i} = W_{ijkl, i} + u_i(x, t)W_{ijkl}
\end{equation}
Since, for $n=4$, we have $dV_{\tilde{g}(t)}= e^{4u(x, t)} dV_g$, we obtain
\begin{align}\nonumber
  \mathfrak{D}(\tilde{g}(t)) &=\operatorname{Vol}_{\tilde{g}(t)}(M)^{\frac12} \int_M |\delta_{\tilde{g}(t)} W_{\tilde{g}(t)}|_{\tilde{g}(t)}^2\, dV_{\tilde{g}(t)} \\\label{eq-pippo}
  &=\left(\int_Me^{4u(x, t)} dV_g\right)^{\frac12} \int_M e^{-2u(x, t)}\pa{\abs{\delta_{g} W_g}^2 + \abs{u_s(x, t) W_{sijk}}_{g}^2 + 2u_s(x, t)W_{sijk}W_{pijk, p} }\,dV_g.
\end{align}
 Letting $V:= \operatorname{Vol}_g(M)$, a simple computation shows that
 \begin{align*}
   \frac{d\mathfrak{D}(\tilde{g}(t))}{dt}\bigg|_{t=0}&=\frac{1}{2}\pa{\int_M e^{4u(x, 0)}}^{-1/2}\pa{\int_M 4U\,dV_g}\int_M\abs{\delta_gW_g}_g^2\,dV_g \\&+\sqrt{V}\int_M\pa{-2U\abs{\delta_gW_g}_g^2+2U_sW_{sijk}W_{pijk, p}}\,dV_g
    \\&=2\set{\int_M\sq{V^{-1/2}\pa{\int_M \abs{\delta_gW}_g^2\,dV_g}-V^{1/2}\abs{\delta_gW}_g^2-V^{1/2}\pa{W_{sijk, s}W_{pijk}}_p} U\,dV_g};
 \end{align*}
 thus we deduce that $g$ is a critical point in the conformal class of the functional, i.e. 
 \begin{equation*}
 \frac{d\mathfrak{D}(\tilde{g}(t))}{dt}\bigg|_{t=0}=0 \quad\quad \forall U \in \cinf \,,
\end{equation*}
if and only if 
\begin{equation}\label{eq_ELprel}
  W_{sijk, sp}W_{pijk} = \frac{1}{\operatorname{Vol}_g(M)}\int_M \abs{\delta_gW}_g^2\,dV_g -2\abs{\delta_gW}_g^2.
\end{equation}
Now, exploiting the algebraic properties of the curvature in dimension four, we show that  \eqref{eq_ELprel} is equivalent to the condition defining weak harmonic Weyl metrics (see Definition \ref{def_WHW}).

\begin{proposition}
  Let $M^4$ be a closed smooth manifold. Then a metric $g$ is critical in the conformal class for the functional $\mathfrak{D}(g)$  if and only if $g$ is a weak harmonic Weyl metric, i.e. it satisfies the formula
 \begin{equation}\label{EQ_EulerLagrange}
\frac{1}{2}\Delta |W|^{2} = |\nabla W|^{2}+\frac{1}{2}R|W|^{2}-3\,W_{ijkl}W_{ijpq}W_{klpq}-8|\delta W|^{2}+\frac{4}{\operatorname{Vol}(M)}\int_M|\delta W|^{2}\,dV \,.
\end{equation}
\end{proposition}
\begin{proof}
  First we observe that equation \eqref{eq_ELprel} holds separately for the self dual and the anti-self dual part of Weyl, i.e.
  \[
   W^\pm_{sijk, sp}W^\pm_{pijk} = \frac{1}{\operatorname{Vol}(M)}\int_M \abs{\delta W^\pm}^2\,dV -2\abs{\delta W^\pm}^2.
  \]
  We now perform our computation for the self dual part $W^+$. Since  $W^+_{sijk}W^+_{pijk} = \frac{1}{4}\abs{W^+}\delta_{sp}$ by equation \eqref{eq-ww}, we have
  \begin{equation}
    \pa{W^+_{sijk}W^+_{pijk}}_{ps} = \frac{1}{4}\Delta\abs{W^+}^2.
  \end{equation}
On the other hand,
\begin{align*}
    \pa{W^+_{sijk}W^+_{pijk}}_{ps} &= W^+_{sijk, ps}W^+_{pijk}+W^+_{jkis, p}W^+_{pijk, s}+\abs{\delta W^+}^2+W^+_{pijk}W^+_{sijk, sp},
  \end{align*}
  which implies
  \[
  2W^+_{sijk, sp}W^+_{pijk} = \Theta + \frac{1}{4}\Delta\abs{W^+}-W^+_{jkis, p}W^+_{pijk, s}-\abs{\delta W^+}^2,
  \]
  where $\Theta := W^+_{sijk, sp}W^+_{pijk}-W^+_{sijk, ps}W^+_{pijk}$; thus, using \eqref{eq_ELprel} and Lemma \ref{l-urka}
  we deduce
  \begin{equation}\label{eq_urca}
    \frac{1}{2}\Delta\abs{W^+}^2= \abs{\nabla W^+}^2 -8\abs{\delta W^+}^2 +\frac{4}{V}\int_M \abs{\delta W^+}^2\,dV -2\Theta.
  \end{equation}
Now, using formulas (3.16)-(3.22) and Lemma 3.4 in \cite{cgy2}, a computation shows that
\[
2\Theta= 3(W^+)^3 - \frac{R}{2}\abs{W^+}^2,
\]
thus completing the proof for $W^+$. The same computation gives the formula for $W^-$ and, by orthogonality, \eqref{EQ_EulerLagrange} follows.
\end{proof}

\


\subsection{The PDE for the conformal factor}

Let now $g_0$ be a fixed Riemannian metric on $M^4$ and let $v:=e^{-u}$ for some $u\in C^{\infty}(M^4)$. Then, equation \eqref{eq-pippo} yields 
$$
\mathfrak{D}(v):=\mathfrak{D}(v^{-2}g_0)= \left(\int_M v^{-4} dV\right)^{\frac12} \int_M \pa{\frac{1}{4}|W|^{2}|\nabla v|^{2}+\abs{\delta W}^2 v^{2} -  (v)^{2}_{s}\, W_{sijk}W_{pijk, p} }\,dV
$$
where all the geometric quontities are referred to the fixed metric $g_0$.
Clearly we have
$$
\mathcal{D}(M,[g_0]) = \inf_{0<v\in C^{\infty}(M)} \mathfrak{D}(v)\,.
$$
To simplify the notation, now let 
$$
a:=\frac{1}{4}|W|^{2}, \quad B_{s}:= W_{sijk}W_{pijk, p}\,\,\quad\text{and}\quad\,\, c_0:=|\delta W |^{2} \,.
$$
Then the previous relation rewrites as
$$
\mathfrak{D}(v) =\left(\int_M v^{-4} dV\right)^{\frac12} \int_M \pa{a |\nabla v|^{2}+c_0\, v^{2} -  \langle B, \nabla (v)^{2}\rangle }\,dV\,.
$$
Imposing that $v$ is critical for the functional $v\mapsto\mathfrak{D}(v)$, i.e.
$$
\frac{d\mathfrak{D}(v+t\varphi)}{dt}\bigg|_{t=0}=0 \quad\quad\forall\, 0<\varphi\in C^{\infty}(M)\,,
$$
we obtain the Euler-Lagrange equation
\begin{equation}\label{theeq}
-\diver\pa{a \nabla v}+ c\, v = \frac{\lambda(v)}{v^{5}}
\end{equation}
with
$$
c=c_0+\diver{B}\quad\quad\text{and}\quad\quad\lambda(v):= \mathfrak{D}(v)\left(\int_M v^{-4} dV\right)^{-3/2}\,.
$$
In particular, the metric $v^{-2}g_0$ is a weak harmonic Weyl metric, i.e. it satisfies \eqref{EQ_EulerLagrange}.

\

\section{Existence results and proof of Theorem \ref{t-main}} \label{sec-ex}

\subsection{Proof of Theorem \ref{t-main}}

Consider the uniformly elliptic self-adjoint operator
\begin{equation}\label{ope}
L v := -\diver\pa{a\nabla v }+ c\, v
\end{equation}
whit $a\in C^\infty(M)$, $a>0$ and $c\in C^\infty(M)$. Note that no sign conditions are required on the coefficient $c$.  We assume that the first eigenvalue of $L$, i.e.
$$
\lambda_1 := \inf_{u\in H^1(M),\, u\not\equiv 0} \mathfrak{R}(u)=\inf_{u\in H^1(M),\, u\not\equiv 0} \frac{ \int_M \left(a |\nabla u|^2 + c\,u^2\right)\,dV}{\int_{M} u^2\,dV}\,,
$$
is non-negative. We will show the existence of positive solutions to the equation
$$
L v = \frac{\lambda(v)}{v^5}
$$
with
$$
\lambda(v)= \mathfrak{D}(v)\left(\int_M v^{-4} dV\right)^{-\frac{3}{2}}, \,\,\,
\mathfrak{D}(v) =\pa{\int_M v^{-4} dV}^{\frac{1}{2}}\int_M \pa{a|\nabla v|^2+c\, v^{2}  }\,dV
$$
and that any two such solutions are proportional. More precisely, the solutions $v$ that we find satisfy
$$
\mathfrak{D}(v) =\hat{\mathcal D}
$$
where
\begin{equation}\label{gianni}
\hat{\mathcal D}:=\inf_{0<u\in C^\infty(M)} \mathfrak{D}(u)\,.
\end{equation}

Now let $(M^4, g_0)$ be a closed Riemannian manifold where $g_0$ is the metric of Aubin (see the Introduction) satisfying 
\begin{equation}\label{ass}
|W_{g_0}|^2_{g_0}>0\quad\text{on}\quad M\,.
\end{equation}

Note that the previous assumptions on the operator $L$ are satisfied in the case of the geometric one defined in Section \ref{sec-eul} with the choice $a:=\frac{1}{4}|W_{g_0}|{g_0}^{2}$ and $c_{g_0}=(c_0)_{g_0}+\diver_{g_0}{B}_{g_0}$. Therefore the conformal metrics $v^{-2} g_0$ have weak harmonic Weyl curvature and the proof of Theorem \ref{t-main} is completed.

\subsection{Preliminary results} From now on, all geometric quantities are referred to the metric $g_0$ and we omit to write their dependence. Let
$$
H(M):=\left\{u\in H^1(M): u>0 \,\,\text{a.e. and}\,\int_M u^{-4}\,dV<\infty \right\}
$$
and define
$$
\mathcal{D}:=\inf_{u\in H(M)} \mathfrak{D}(u)\,.
$$
By standard elliptic theory, there exists a smooth, positive, first eigenfunction $\varphi_1$ of $L$, i.e. a solution of
$$
L \varphi_1 = \lambda_1 \varphi_1\,.
$$
Note that $\mathfrak{R}(\varphi_1)=\lambda_1$. We have the following weak maximum principle.

\begin{lemma}\label{l-wmp} Let $\lambda_1>0$. Under the previous assumptions, if $u\in H^1(M)$ satisfies $Lu\geq 0$ in the weak sense, then $u\geq 0$ a.e. on $M$.
\end{lemma}

Moreover, using Lemma \ref{l-wmp}, one can prove the following strong maximum principle.

\begin{lemma}\label{l-smp} Let $\lambda_1>0$. Under the previous assumptions, if $u\in H^1(M)$ satisfies $Lu\geq 0$ in the weak sense, then either $u =0$ a.e. on $M$ or $\operatorname{essinf}_M u >0$.
\end{lemma}

We have a two-sided estimate on $\mathcal{D}$ in terms of $\lambda_1$

\begin{lemma}\label{lem1} Under the previous assumptions, we have
$$
\operatorname{Vol}(M)^{-\frac{1}{2}}\,\lambda_1 \leq \mathcal{D} \leq \frac{\int_M \varphi_1^2\,dV}{\left(\int_M \varphi_1^{-4}\,dV\right)^{\frac{1}{2}}}\,\lambda_1 \,.
$$
\end{lemma}
\begin{proof}
By Jensen inequality, for every $u\in H(M)$
$$
\frac{1}{\int_M u^2\,dV}\leq \left(\int_M u^{-4}\,dV\right)^{\frac{1}{2}} \operatorname{Vol}(M)^{\frac{1}{2}}\,.
$$
Then
$$
\lambda_1 \leq \mathfrak{R}(u) \leq \mathfrak{D}(u) \operatorname{Vol}(M)^{\frac{1}{2}}
$$
and the first inequality follows. Moreover, for every $u\in H(M)$ we have
$$
\mathfrak{D}(u) = \mathfrak{R}(u) \frac{\int_M u^2\,dV}{\left(\int_M u^{-4}\,dV\right)^{\frac{1}{2}}}\,.
$$
Then
$$
\mathcal{D}\leq \mathfrak{D}(\varphi_1) = \mathfrak{R}(\varphi_1) \frac{\int_M \varphi_1^2\,dV}{\left(\int_M \varphi_1^{-4}\,dV\right)^{\frac{1}{2}}}=\frac{\int_M \varphi_1^2\,dV}{\left(\int_M \varphi_1^{-4}\,dV\right)^{\frac{1}{2}}}\,\lambda_1  \,.
$$
\end{proof}

Consequently, by Lemma \ref{l-wmp} and Lemma \ref{l-smp}, maximum principles hold whenever $\mathcal{D}>0$ and $\mathcal{D}=0$ if and only if $\lambda_1=0$.

\subsection{Existence}
In this subsection we prove that the functional $u\mapsto\mathfrak{D}(u)$ admits a minimum $v$ in $H(M)$, that $v$ satisfies the associated Euler-Lagrange equation and it is smooth.
\begin{lemma}\label{lem2}
Suppose that $u\in H(M)$ and that $\mathfrak D(u)=\mathcal D>0$. Then
$L u\geq 0$ in the weak sense, i.e.
\[\int_M \left\{a\langle\nabla u, \nabla\varphi \rangle + c  u \varphi \right\} dV\geq 0\quad \textrm{for any}\;\; \varphi\in C^1(M), \varphi\geq 0\,.\]
\end{lemma}
\begin{proof}
By contradiction, assume that there exists $\varphi\in C^1(M), \varphi\geq 0$ such that
\[\int_M \left\{a\langle\nabla u, \nabla\varphi \rangle + c  u \varphi \right\} dV < 0\,. \]
For every $v\in H^1(M)$ define
\[Q(v):=\int_M \left\{a|\nabla v|^2 + c v^2  \right\}dV\,.\]
Note that
$$Q(v) \geq \lambda_1\|v\|_{L^2}^2\,.$$
Take any $t\in \mathbb R$ with $|t|$ small enough. We have that
\begin{equation*}
\begin{aligned}
&\mathfrak D(u+t\varphi)-\mathfrak D(u)\\
&=\left(\int_M (u+t\varphi)^{-4} dV \right)^{\frac 1 2}Q(u+t\varphi)-\left(\int_M u^{-4} dV\right)^{\frac 1 2}Q(u)\\
&=\left[\left(\int_M (u+t\varphi)^{-4} dV \right)^{\frac 1 2}- \left(\int_M u^{-4} dV\right)^{\frac 1 2}\right]Q(u+t\varphi)
\\
&+\left(\int_M u^{-4} dV\right)^{\frac 1 2}\left[Q(u+t\varphi)-Q(u)\right]\,.
\end{aligned}
\end{equation*}
Furthermore,
\[Q(u+t\varphi)\geq 0\,,\]
\[\left((u+t\varphi)^{-4} dV \right)^{\frac 1 2}-\left(u^{-4}dV \right)^{\frac 1 2}\quad\textrm{for any}\;\; t>0\,,, \]
and
\[Q(u+t\varphi)-Q(u)=Q'(u)[\varphi]+ o(t)\quad \textrm{as}\;\; t\to 0,\]
where
\[Q'(u)[\varphi]=\int_M\{a\langle \nabla u, \nabla\varphi + c u \varphi \}dV >0\,.\]
Thus, for $t>0$ sufficiently small,
\[Q(u+t\varphi)-Q(u) \leq \left(\int_M u^{-4} dV \right)^{\frac 1 2}\left\{Q'(u)[\varphi]t + o(t) \right\}<0.\]
So,
\[\mathfrak D(u+t\varphi)<\mathfrak D(u)\]
with $u+t\varphi>0$ a.e., $u+t\varphi\in H(M)$. This is a contradiction, since
\[\mathfrak D(u)=\mathcal D\,.\]
 \end{proof}

\begin{cor}\label{cor3}
Suppose that $u\in H(M)$ and that $\mathfrak D(u)=\mathcal D>0$. Then $\operatorname{essinf}_M u >0$.
\end{cor}
\begin{proof}
The thesis follows from Lemmas \ref{l-smp} and \ref{lem2}.
\end{proof}

\begin{theorem}\label{thmexi}
There exists $v\in C^{\infty}(M), v>0$ such that
\[\mathfrak D(v)=\hat{\mathcal D}\,.\]
Moreover $v$ satisfies
\[Lu = \mathcal D \left(\int_M v^{-4} dV \right)^{\frac 3 2} v^{-5}\,.\]
\end{theorem}
\begin{proof}
First we suppose that
\[{\mathcal D}=0\,.\]
By Lemma \ref{lem1}, $\lambda_1=0$. Moreover,
\[\mathfrak D(\varphi_1)=\frac{\int_M \varphi_1^2 dV}{\left(\varphi_1^{-4} dV\right)^{\frac 1 2}}\mathfrak R(\varphi_1)
=\frac{\lambda_1 \int_M \varphi_1^2 dV}{\left(\varphi_1^{-4} dV\right)^{\frac 1 2}}=0={\mathcal D}\,.\]
Since $\varphi_1\in C^{\infty}(M)$, we have that also $\hat{\mathcal D}=0$. Hence
\[\mathfrak D(\varphi_1)={\mathcal D}=\hat{\mathcal D}=0\,.\]
From now on we suppose that $\mathcal D>0.$ Let $\{v_n\}_{n\in \mathbb N}\subset H(M)$ be a sequence of functions such that
$\mathfrak D(v_n)\to \mathcal D$. Since the functional $\mathfrak D$ is scaling invariant, without loss of generality, we can assume that $\int_M v_n^{-4} dV=1$. Since $\mathcal D>0$, in view of Lemma \ref{lem1}, we have that $\lambda_1>0.$ In addition,
\[ \int_M\{ a|\nabla v_n|^2 + c v_n^2\} dV \geq \lambda_1 \int_M v_n^2 dV\,.\]
Clearly, for any $n\in \mathbb N$ sufficiently large,
\[ 0<\mathcal D \leq \mathfrak D(v_n) \leq \mathcal D+1\,.\]
Hence
\begin{equation*}
\begin{aligned}
\lambda_1\int_M v_n^2 dV & \\& \leq
 \int_M \{a|\nabla v_n|^2 + c v_n^2 \} dV \\
&=\left(v_n^{-4} dV \right)^{\frac 1 2}\int_M\{ a|\nabla v_n|^2 + c v_n^2\} dV \\
&=\mathfrak D(v_n)\leq \mathcal D+1\,.
\end{aligned}
\end{equation*}
So, $\{v_n\}$ is bounded in $L^2(M)$. Moreover,  for any $n\in \mathbb N$ sufficiently large,
\begin{equation}
\begin{aligned}
\min_M a \int_M |\nabla v_n|^2 dV &\\ & \leq
\int_M a|\nabla v_n|^2 dV = \mathfrak D(v_n) - \int_M c v_n^2 dV\\
& \leq \mathfrak D(v_n) + \|v \|_{L^{\infty}}\|v_n\|_{L^2}^2\\
& \leq \mathcal D +1 + \|v \|_{L^{\infty}}\frac{\mathcal D +1}{\lambda_1}\,.
\end{aligned}
\end{equation}
So, $\{\nabla v_n\}$ is bounded in  $L^2(M)$, and $\{v_n\}$ is bounded in $H^1(M)$. Consequently, there exist a subsequence of $\{v_n\}$, which will be still denoted by $\{v_n\}$, and a function $v\in H^1(M)$ such that
\[ v_n \mathop{\rightharpoonup}_{n\to \infty} v\quad \textrm{in}\;\; H^1(M)\,,\]
\[ v_n \mathop{\rightarrow}_{n\to \infty} v\quad \textrm{in}\;\; L^2(M)\,,\]
\[ v_n \mathop{\rightarrow}_{n\to \infty} v\quad \textrm{a.e. in}\;\;M\,.\]
Therefore,
\[v_n^{-4}\mathop{\rightarrow}_{n\to \infty} v^{-4}\quad \textrm{a.e. in}\;\;M\,;\]
here we have assumed that $v_n, v: M \to [0, +\infty]$ and $\frac 1{\infty}=0, \frac 1 0 =\infty.$
By Fatou's lemma,
\[\int_M v^{-4} dV = \int_M \mathop{\operatorname{liminf}}_{n\to\infty}  v_n^{-4} dV \leq \mathop{\operatorname{liminf}}_{n\to\infty} \int_M v_n^{-4} dV =1\,.\]
Thus
\[\int_M v^{-4} dV < +\infty\,.\]
This implies that $v>0$ a.e. in $M$. In fact, if $v=0$ in a set of positive measure, since $v\geq 0$, we get
$\int_M v^{-4} dV =\infty$, a contradiction. Hence
$v\in H^1(M), v<0$ a.e. in $M$, $\int_M v^{-4} dV\leq 1.$ Using the fact that $\displaystyle{v_n \mathop{\rightharpoonup}_{n\to \infty} v}$ in $H^1(M)$ and
$\displaystyle{v_n \mathop{\rightarrow}_{n\to \infty} v}$ in $L^2(M)$, we can infer that
\begin{equation*}
\begin{aligned}
\mathcal D\leq \mathfrak D(v)\\ &= \left(v^{-4} dV \right)^{\frac 1 2}\int_M\{ a|\nabla v|^2 + c v^2 \}dV \\
& \leq \int_M\{a|\nabla v|^2 + c v^2 \}dV \\
&\leq \mathop{\operatorname{liminf}}_{n\to\infty} \int_M \{a|\nabla v_n|^2 + c v_n^2 \}dV=
\mathop{\operatorname{liminf}}_{n\to\infty} \mathfrak D(v_n) =\mathcal D\,.
\end{aligned}
\end{equation*}
So,
\[\mathfrak D(v)=\mathcal D>0\,.\]
From Lemma \ref{l-smp} it follows that $\operatorname{essinf} v>0$. Take any $\varphi\in C^1(M)$. Since $\mathfrak D(v)=\mathcal D$,
we get
\[\frac{d}{dt} \left[\mathfrak D(v + t\varphi) \right]|_{t=0}=0\,.\]
Consequently, for any $\varphi\in C^1(M)$, we have
\[\int_M \{a \langle \nabla v, \nabla \varphi + c v \varphi\} d V = \int_M \mathcal D\left(v^{-4}\right)^{-\frac 3 2} dV \int_M \frac{\varphi}{v^5} dV\,.\]
Thus,
\[Lu = \mathcal D \left(\int_M v^{-4} dV \right)^{\frac 3 2} v^{-5}=:f\quad \textrm{weakly\, in}\;\; M\,. \]
Since $\operatorname{essinf} v>0$, we have that $f\in L^{\infty}(M)$. Therefore, by standard elliptic regularity theory,
$v\in C^{\infty}(M), v>0$ in $M$. We can therefore infer that
\[\hat{\mathcal D}\leq \mathfrak D(v)=\mathcal D\leq \hat{\mathcal D}.\]
Hence,
$$\mathfrak D(v)=\mathcal D\,.$$
This completes the proof.
\end{proof}

\begin{rem} From the proof of Theorem \ref{thmexi} we can deduce that  $\mathcal D=\hat{\mathcal D}$, so
\begin{equation}\label{eq1}
L v =  \hat{\mathcal D}\left(\int_M v^{-4} dV\right)^{-\frac 3 2}  v^{-5} \quad \textrm{in}\;\; M\,.
\end{equation}
\end{rem}

\subsection{Uniqueness}
Observe that equation \eqref{eq1} is scaling invariant, in the sense that if $u_1$ solves \eqref{eq1}, then $u_1:=\beta u_2$, with $\beta\in \mathbb R^+,$ satisfies
\[L u_2 =  \hat{\mathcal D}\left(\int_M u_2^{-4} \right)^{\frac 3 2} u_2^{-5} dV \quad \textrm{in}\;\; M\,.\]
Therefore, uniqueness for equation \eqref{eq1} does not hold. However, we have the following result.

\begin{theorem}\label{thmuni1}
Suppose that both $u_1$ and $u_2$ are solutions of equation \eqref{eq1} and that $u_1>0, u_2>0$ in $M$. Then there exists $\beta\in \mathbb R^+$  such that
\[ u_1=\beta u_2\quad \textrm{in}\;\; M\,.\]
\end{theorem}
\begin{proof}
Let
$$\mu:=\left(\int_M u_1^{-4} dV\right)^{\frac 1 4},\quad \gamma:=\left(\int_M u_2^{-4} dV\right)^{\frac 1 4}.$$
So, the functions
\[\psi:= \mu u_1, \quad w:=\gamma u_2\]
satisfy
\[\int_M \psi^{-4} dV=\int_M w^{-4} dV=1\,.\]
Then
\[ L \psi = \hat{\mathcal D} \psi^{-5} \quad \textrm{in}\;\; M\,,\]
\[ L w = \hat{\mathcal D} w^{-5} \quad \textrm{in}\;\; M\,.\]
We choose $\alpha>0$ such that
\[\psi -\alpha w\geq 0\quad \textrm{and}\;\; \min_M \{\psi-\alpha w \} =0\,.\]
Since $M$ is compact, we can find a minimum point $x_0\in M$ of the continuous function $\psi-\alpha w$, so that $\psi(x_0)=\alpha w(x_0)\,.$
First assume that $\hat{\mathcal D}>0$. We have that
\[L (\psi-\alpha w)=\hat{\mathcal D}(\psi^{-5}-\alpha w^{-5})\quad \textrm{in}\;\; M\,.\]
In particular, at $x_0$ we obtain
$$0\geq \hat{\mathcal D}(\psi^{-5}(x_0)-\alpha w^{-5}(x_0))=\frac{\hat{\mathcal D}(1-\alpha^6)}{\alpha^5 w^5(x_0)}\,.$$
This yields $\alpha\geq 1$, and so,
\[\psi\geq \alpha w\geq w\quad \textrm{in}\;\; M\,.\]
By repeating the same argument interchanging the role of $\psi$ and $w$, we get
\[w\geq \psi\quad \textrm{in}\;\; M\,.\]
Hence
\[\mu u_1=\psi=w=\gamma u_2\quad \textrm{in}\;\; M\,.\]
Thus, we obtain the thesis with $\beta=\frac{\gamma}{\mu}.$

Now, assume that $\hat{\mathcal D}=0.$ Since $\psi-\alpha w\geq 0$, if we take $M>\|c\|_{L^{\infty}}$, then we have
\[L(\psi-\alpha w) + M (\psi-\alpha w)\geq 0\quad \textrm{in}\;\; M,\]
and $\min_M\{\psi-\alpha w\}=0\,.$ Observe that
\[\lambda_1(L+M \operatorname{Id})= \lambda_1(L)+M=M>0.\]
Thus, by Lemma \ref{l-smp} applied to the operator $L+M Id$ for the function $\psi-\alpha w$, we obtain $\psi=\alpha w$.
Therefore,
\[u_1=\frac{\gamma}{\mu}\alpha u_2\quad \textrm{in}\;\; M\,.\]
The proof is now complete.
\end{proof}

\begin{rem}
The proof of Theorem \ref{thmuni1} in the case $\hat{\mathcal D}=0$ is equivalent to the proof that $\lambda_1(L)$ is simple.
\end{rem}

Consider equation
\begin{equation}\label{eq2}
L u = \mathfrak{D}(u)\left(\int_M u^{-4} dV \right)^{-\frac 3 2} u^{-5} \quad \textrm{in}\;\; M\,.
\end{equation}
Observe that $\mathfrak{D}(u)=\mathfrak D(\beta u)$ for any $\beta\in \mathbb R^+$. Furthermore, equation \eqref{eq2}
is scaling invariant as before, so uniqueness for equation \eqref{eq2} does not hold. However, we have the following result.

\begin{theorem}\label{thmuni2}
Suppose that both $u_1$ and $u_2$ are solutions to equation \eqref{eq2}, and that $u_1>0, u_2>0$ in $M$. Then there exists $\beta\in \mathbb R^+$  such that
\[ u_1=\beta u_2\quad \textrm{in}\;\; M\,.\]
\end{theorem}
\begin{proof}
First assume that $\mathfrak D(u_1)>0, \mathfrak D(u_2)>0$. Let
$$\mu:=\left(\int_M u_1^{-4} dV\right)^{\frac 1 4}\mathfrak D(u_1),\quad \gamma:=\left(\int_M u_2^{-4} dV\right)^{\frac 1 4}\mathfrak D(u_2).$$
So, the functions
\[\psi:= \mu u_1, \quad w:=\gamma u_2\]
satisfy
\[ L \psi =  \psi^{-5} \quad \textrm{in}\;\; M\,,\]
\[ L w = w^{-5} \quad \textrm{in}\;\; M\,.\]
Hence the conclusion follows as in the proof of Theorem \ref{thmuni1}, when $\hat{\mathcal D}>0.$ Also, in the case $\mathfrak D(u_1)=\mathfrak D(u_2)=0,$ the thesis is obtained by the same arguments as in the proof of Theorem \ref{thmuni1}, when $\hat{\mathcal D}=0.$

We claim that the case $\mathfrak D(u_1)>0$ and $\mathfrak D(u_2)=0$ cannot happen. Indeed, by contradiction assume that $\mathfrak D(u_1)>0$ and $\mathfrak D(u_2)=0$. Define
\[\psi:= \mu u_1, \quad w:= u_2\,.\]
We choose $\alpha>0$ such that
\[\psi -\alpha w\geq 0\quad \textrm{and}\;\; \min_M \{\psi-\alpha w \} =0\,.\]
Since $M$ is compact, we can find a minimum point $x_0\in M$ of the continuous function $\psi-\alpha w$, so that $\psi(x_0)=\alpha w(x_0)\,.$
We have that
\[L (\psi-\alpha w)=\psi^{-5}\quad \textrm{in}\;\; M\,.\]
In particular, at $x_0$ we obtain
$$0\geq \psi^{-5}(x_0)>0\,.$$
This is a contradiction. The proof is now complete.
\end{proof}

\begin{cor}
Every critical point of the functional $u\mapsto \mathfrak D(u)$, defined in $H(M)$, is a minimum point.
\end{cor}
\begin{proof}
Let $w$ be a critical point of the functional $u\mapsto \mathfrak D(u)$.
Recall that this is equivalent to requiring that $\mathfrak D'(w)=0$, i.e. $w$ is a solution of equation \eqref{eq2}. By Theorem \ref{thmexi}, there exists a minimum point $v$ of the functional $u\mapsto \mathfrak D(u)$, which is a solution of equation \eqref{eq1}. By Theorem \ref{thmuni2} with $u_1=w$ and $u_2=v$ we can infer that $w=\beta v$, for some $\beta>0$. Then
$$\mathfrak D(w)=\mathfrak D(\beta v)=\mathfrak D(v)=\hat{\mathcal D}\,.$$
This is the thesis.
\end{proof}

\subsection{Further results}
For any $\beta>0$ consider equation
\begin{equation}\label{eq4}
L u = \beta u^{-5}\quad \textrm{in}\;\; M\,.
\end{equation}

Let
\[\underline{l}:=\min_M \varphi_1, \quad \overline{l}:=\max_M \varphi_1\,.\]

\begin{proposition}\label{prop1}
Assume that $\lambda_1>0$ and $\beta>0$. Then there exists a solution $u\in C^{\infty}(M)$ of equation \eqref{eq4} such that
\[ \frac{\underline l}{\overline l}\left( \frac{\beta}{\lambda_1}\right)^{\frac 1 6} \leq u \leq
\frac{\overline l}{\underline l}\left( \frac{\beta}{\lambda_1}\right)^{\frac 1 6}\quad \textrm{in}\;\; M\,.\]
Moreover, if $v>0$ is any solution of equation \eqref{eq4}, then $v=u$ in $M$\,.
\end{proposition}

\begin{proof} Define
\[\underline u:= \underline \alpha \varphi_1, \quad \overline u:= \overline \alpha \varphi_1,\]
where $\underline \alpha, \overline \alpha$ are positive constants to be chosen.
It is easily seen that if $\underline \alpha\leq \frac{\beta^{\frac16}}{\overline l \lambda_1^{\frac 1 6}}$, then
$\underline u$ is a subsolution of equation \eqref{eq4}, that is
\[ L \underline u \leq \beta \underline u^{-5}\quad \textrm{in}\;\; M\,.\]
In fact,
\[L \underline u=\lambda_1 \underline \alpha \varphi_1\leq \beta \underline u^{-5}=\beta \underline \alpha^{-5} \varphi_1^{-5}\quad \textrm{in}\;\; M\,,\]
provided that $\underline \alpha\leq \frac{\beta^{\frac 16}}{\overline l \lambda_1^{\frac 1 6}}$.
It is similarly seen that if $\overline \alpha\geq \frac{\beta^{\frac 16}}{\underline l \lambda_1^{\frac 1 6}}$, then
$\overline u$ is a supersolution of equation \eqref{eq4}, that is
\[ L \overline u \geq \beta\, \overline u^{-5}\quad \textrm{in}\;\; M\,.\]
Clearly, $0<\underline \alpha \leq  \overline \alpha.$ Define
\[\mathcal L u:= -\diver\pa{a\nabla u }\,.\]
Hence equation \eqref{eq2} is equivalent to equation
\begin{equation}\label{eq5}
\mathcal L u = f(u)\quad \textrm{in}\;\; M\,,
\end{equation}
where $f(u):= - c u +  u^{-5}$. We have shown that $\underline u$ is a subsolution of equation \eqref{eq5}, while $\overline u$ is a supersolution of equation \eqref{eq5}. Moreover,
\[0< \underline{\alpha}\,\underline{l}\leq \underline u\leq \overline u\leq \overline\alpha\overline{l}\quad \textrm{in}\;\; M\,,\]
and $f\in C^1([\underline \alpha \,\underline{l}, \overline\alpha\overline{l}])$. Hence by the standard sub-- and supersolutions method, we can infer that there exists a weak solution to equation \eqref{eq5}, and hence to equation \eqref{eq4}, satisfying
\[ \underline u \leq u \leq \overline u\quad \textrm{in}\;\; M\,.\]
By standard regularity theory it follows that $u\in C^{\infty}(M).$ Moreover, by the same arguments as in the proof of Theorem \ref{thmuni1} when $\hat{\mathcal D}>0$ we can infer that
if $v>0$ is any solution of equation \eqref{eq4}, then $v=u$. This completes the proof.
\end{proof}

\begin{proposition}\label{prop2}
Suppose that $\hat{\mathcal D}>0$. Let $v$ be a solution of equation \eqref{eq1}. Then
\[ \frac{\underline l}{\overline l}\left( \frac{\hat{\mathcal D}}{\lambda_1}\right)^{\frac 1 6} \leq v \left(\int_M v^{-4} dV \right)^{\frac 1 4} \leq
\frac{\overline l}{\underline l}\left( \frac{\hat{\mathcal D}}{\lambda_1}\right)^{\frac 1 6}\quad \textrm{in}\;\; M\,. \]
\end{proposition}
\begin{proof}
Let $v$ be a solution of equation \eqref{eq1}. So, $v$ is also a solution of equation \eqref{eq4} with $\beta=\hat{\mathcal D}\left(\int_M v^{-4} dV\right)^{-\frac 3 2}$. Hence, by Proposition \ref{prop1}, the thesis follows.
\end{proof}

\

\section{A quantitative improved Kato inequality}\label{sec-kato}

We recall that, given any tensor $T$, at every point where $|T|\neq 0$, one has the classical Kato inequality
$$
|\nabla T|^2 \geq |\nabla |T||^2 \,.
$$
It was proved by Gursky and Lebrun \cite{gurleb}, that on a four manifold $(M^4,g)$ with half harmonic Weyl metric, i.e. $\delta W^\pm \equiv 0$, there holds
$$
|\nabla W^{\pm}|^{2}\geq \frac{5}{3} \big|\nabla |W^{\pm}|\big|^{2}
$$
if $|W^\pm|\neq 0$. In this section we prove a new quantitative version of the classical Kato inequality for the Weyl tensors $W^\pm$.  In particular, we recover the sharp Kato inequality established in \cite{gurleb}.

\begin{lemma}\label{l-IKI} Let $(M^{4},g)$ be a four dimensional Riemannian manifold. Then at a point where $|W^{\pm}|\neq 0$ it holds
$$
|\nabla W^{\pm}|^{2}\geq k \big|\nabla |W^{\pm}|\big|^{2}-\frac{8(k-1)}{(5-3k)}|\delta W^{\pm}|^{2}
$$
for every $k\in\left[0,\frac{5}{3}\right)$. In particular, if $\delta W^{\pm}\equiv 0$, then at a point where $|W^{\pm}|\neq 0$, it holds
$$
|\nabla W^{\pm}|^{2}\geq \frac{5}{3} \big|\nabla |W^{\pm}|\big|^{2}\,.
$$
\end{lemma}
\begin{rem}\label{rem127} 
As it will be clear from the proof, in the case $k=0$, the inequality holds on the whole $M$, even at points where $|W^\pm|=0$.
\end{rem}
\begin{proof}
We perform our computations for the self-dual case; first recall that (see equation \eqref{eq-nqder})
\begin{equation*}
\Vert\nabla W^{+}\Vert^{2} = |d\lambda^{+}|^{2}+|d\mu^{+}|^{2}+|d\nu^{+}|^{2}  +2(\mu^{+}-\nu^{+})^{2}|a^{+}|^{2}+2(\lambda^{+}-\nu^{+})^{2}|b^{+}|^{2}+ 2(\lambda^{+}-\mu^{+})^{2}|c^{+}|^{2}.
\end{equation*}
In the rest of the proof we omit the   ``$+$'' on $\lambda$, $\mu$, $\nu$, $\omega$, $\eta$, $\theta$ and $a$, $b$, $c$ for the sake of simplicity.
We set $\bar a := (\mu-\nu)a$, $\bar b := (\lambda -\nu)b$ and $\bar c := (\lambda-\mu)c$; we also define
$$
X_{j}:=-\omega_{ij}\bar a_{i},\quad Y_{j}:=\eta_{ij}\bar b_{i},\quad Z_{j}:=-\theta_{ij}\bar c_{i} \,.
$$
Then, from the quaternionic structure, we get
$$
|X|^{2}=|\bar a|^{2},\quad |Y|^{2}=|\bar b|^{2}, \quad |Z|^{2} = |\bar c|^{2}
$$
and
$$
\langle X,Y\rangle = - \theta_{ij} \bar b_{i}\bar a_{j},\quad \langle X,Z\rangle = - \eta_{ij}\bar c_{i}\bar a_{j}, \quad \langle Y,Z\rangle = - \omega_{ij}\bar c_{i}\bar b_{j} \,.
$$
Since $\lambda+\mu+\nu=0$, $\Vert W\Vert^{2}=\frac{1}{4}\abs{ W}^2$  and $\Vert\nabla W\Vert^{2}=\frac{1}{4}\abs{\nabla W}^2$ we have
\begin{equation}\label{eq_prelmodW}
\abs{\nabla W^{+}}^2 = 8\pa{|d\lambda|^{2}+\pair{d\lambda, d\nu}+|d\nu|^{2}  +|X|^{2}+|Y|^{2}+ |Z|^{2}}
\end{equation}
and
\begin{align*}
  \abs{\nabla\|W^{+}\|}^2 &= 2\abs{d\pa{\sqrt{\lambda^2+\nu^2+\lambda\nu}}}^2 \\ &=2\abs{\frac{1}{2\sqrt{\lambda^2+\nu^2+\lambda\nu}}\pa{2\lambda d\lambda+2\nu d\nu+\nu d\lambda+\lambda d\nu}}^2 \\ &=\frac{1}{2\pa{\lambda^2+\nu^2+\lambda\nu}}\abs{(2\lambda+\nu)d\lambda+(\lambda+2\nu)d\nu}^2.
\end{align*}
Thus
\begin{equation}\label{eq_ModWeylboh}
   \abs{\nabla\abs{W^{+}}}^2=\frac{2}{\lambda^2+\nu^2+\lambda\nu}\sq{(2\lambda+\nu)^2\abs{d\lambda}^2+(\lambda+2\nu)^2\abs{d\nu}^2+2\pa{2\lambda+\nu}\pa{\lambda+2\nu}d\lambda d\nu}.
\end{equation}

Now, since by \eqref{eq-derder} one has
\begin{align*}
  2W^{+}_{ijpt, t} &= \pa{\lambda_t\omega_{pt}+\bar{c}_t\eta_{pt}-\bar{b}_t\theta_{pt}}\omega_{ij} \\&+\pa{\bar{c}_t\omega_{pt}+\mu_t\eta_{pt}+\bar{a}_t\theta_{pt}}\eta_{ij}\\&+\pa{-\bar{b}_t\omega_{pt}+\bar{a}_t\eta_{pt}+\nu_t\theta_{pt}}\theta_{ij},
\end{align*}
we deduce, after some computations,
\begin{align*}
  \abs{\delta W^{+}}^2 &= \abs{\lambda_t\omega_{pt}+\bar{c}_t\eta_{pt}-\bar{b}_t\theta_{pt}}^2 \\ &+\abs{\bar{c}_t\omega_{pt}+\mu_t\eta_{pt}+\bar{a}_t\theta_{pt}}^2 \\&+\abs{-\bar{b}_t\omega_{pt}+\bar{a}_t\eta_{pt}+\nu_t\theta_{pt}}^2 \\&= \abs{d\lambda}^2 + \abs{d\mu}^2 + \abs{d\nu}^2 +2\abs{\bar{a}}^2+2\abs{\bar{b}}^2+2\abs{\bar{c}}^2 \\ &+2\omega_{st}\pa{-\bar{c}_tb_s + \mu_t\bar{a}_s +\bar{a}_t\nu_s} \\&-2\eta_{st}\pa{-\lambda_t\bar{b}_s + \bar{c}_t\bar{a}_s-\bar{b}_t\nu_s}\\&+2\theta_{st}\pa{\lambda_t\bar{c}_s + \bar{c}_t\mu_s - \bar{b}_t\bar{a}_s},
   \end{align*}
 and thus
 \begin{align}\label{eq_terzadivW}
 \abs{\delta W^{+}}^2
   &= 2|d\lambda|^{2}+2\pair{d\lambda, d\nu}+2|d\nu|^{2}  +2|X|^{2}+2|Y|^{2}+ 2|Z|^{2}\\ \nonumber &+2\pair{d\lambda, X} +2\pair{d\lambda, Y}-4\pair{d\lambda, Z}+4\pair{d\nu, X}-2\pair{d\nu, Y}-2\pair{d\nu, Z}\\ \nonumber &-2\pair{X, Y}-2\pair{X, Z}-2\pair{Y, Z}.
  \end{align}
With respect to the ``formal'' ordered basis $d\lambda$, $d\nu$, $X$, $Y$ and $Z$, we can express the three quantities  in equations \eqref{eq_prelmodW}, \eqref{eq_ModWeylboh} and \eqref{eq_terzadivW} as quadratic forms, with associated matrices given by, respectively,
\[
\mathcal{M}_{\abs{\nabla W^+}^2}=\left[\begin{array}{lllll}8&4&0&0&0\\4&8&0&0&0\\0&0&8&0&0\\0&0&0&8&0\\0&0&0&0&8\end{array}\right],
\]

\[
\mathcal{M}_{\abs{\nabla\abs{W^{+}}}^2}=\frac{2}{\lambda^2+\nu^2+\lambda\nu}\left[\begin{array}{ccccc}(2\lambda+\nu)^2&(2\lambda+\nu)(\lambda+2\nu)&0&0&0\\(2\lambda+\nu)(\lambda+2\nu)&(\lambda+2\nu)^2&0&0&0\\0&0&0&0&0\\0&0&0&0&0\\0&0&0&0&0\end{array}\right],
\]

\[
\mathcal{M}_{\abs{\delta W^{+}}^2}=\left[\begin{array}{ccccc}2&1&1&1&-2\\1&2&2&-1&-1\\1&2&2&-1&-1\\1&-1&-1&2&-1\\-2&-1&-1&-1&2\end{array}\right],
\]

Now we define the quantity $Q := \abs{\nabla W^+}^2 +k_1\abs{\delta W^{+}}^22-k_2\abs{\nabla\abs{W^{+}}}^2$, for some $k_1, k_2\in\mathds{R}$, with associated matrix $\mathcal{Q}= \mathcal{M}_{\abs{\nabla W^+}^2} +k_1\mathcal{M}_{\abs{\nabla\abs{W^{+}}}^2}-k_2\mathcal{M}_{\abs{\delta W^{+}}^2}$; we have then
\[
\mathcal{M}_{Q}=\left[\begin{array}{ccccc}8+2k_1-2k_2\frac{(2\lambda+\nu)^2}{\lambda^2+\nu^2+\lambda\nu}&4+k_1-2k_2\frac{(2\lambda+\nu)(\lambda+2\nu)}{\lambda^2+\nu^2+\lambda\nu}&k_1&k_1&-2k_1\\4+k_1-2k_2\frac{(2\lambda+\nu)(\lambda+2\nu)}{\lambda^2+\nu^2+\lambda\nu}&8+2k_1-2k_2\frac{(\lambda+2\nu)^2}{\lambda^2+\nu^2+\lambda\nu}&2k_1&-k_1&-k_1\\k_1&2k_1&8+2k_1&-k_1&-k_1\\k_1&-k_1&-k_1&8+2k_1&-k_1\\-2k_1&-k_1&-k_1&-k_1&8+2k_1\end{array}\right]
\]
A computation gives that $\operatorname{det}(\mathcal{M}_{Q})=384\pa{8+5k_1}\pa{8+5k_1-8k_2-3k_1k_2}$\,.
Let
$$
k:=k_{2} \quad\quad\text{and}\quad\quad k_{1}:=\frac{8(k-1)}{5-3k}\,.
$$
Thus $\operatorname{det}(\mathcal{M}_{Q})=0$. We claim that the matrix $\mathcal{M}_{Q}$ is non-negative definite. In fact, we can check that the principal minors of order $2,3$ and $4$ have determinants, respectively,
$$
\frac{144(3-k)(1-k)^{2}}{(5-3k)^{2}}\,,
$$
$$
\frac{384(1-k)^{2}\big((3+2k)\lambda^{2}+(3+k)\lambda\nu+(3+k)\nu^{2}\big)}{(5-3k)^{2}(\lambda^{2}+\lambda\nu+\nu^{2})}\,,
$$
and
$$
\frac{3072\, k (1-k)^{2}(2\lambda+\nu)^{2}}{(5-3k)^{2}(\lambda^{2}+\lambda\nu+\nu^{2})}\,.
$$
Since $k\in[0,\frac{5}{3})$, it is easy to see that all these quantities are nonnegative. Moreover, with similar computations, one verify that all the leading minors have non-negative determinants. Thus $\mathcal{M}_{Q}$ is non-negative definite and the inequality is proved.
\end{proof}

\

\section{Rigidity results: proof of Theorem \ref{t-rigidity} and Corollary \ref{c-asf}}\label{sec-rigidity}

In this section we prove Theorem \ref{t-rigidity} and Corollary \ref{c-asf}. Let $(M^{4},g)$ be a closed manifold of dimension four with positive Yamabe invariant, $\mathcal{Y}(M,[g])>0$. Assume that $(M^{4},g)$ is not anti-self-dual, i.e. $W^{+}\not\equiv 0$, and satisfies the pinching
\begin{equation}\label{asss}
\mathcal{W^{+}}(M, [g]) \leq \frac{\alpha^{2}}{6} \mathcal{Y}(M,[g])^{2}\,,
\end{equation}
for some $\alpha\in[0, \frac{5}{9}]$. Obiouvsly, if $\alpha=0$ we have a contradiction. Moreover, the case $\alpha=\frac59$ was already considered in \cite{gur2} (see also \cite{hebvau}). Hence we can assume $\alpha\in(0, \frac{5}{9})$. In order to prove Theorem \ref{t-rigidity}, we will show that
$$
\mathfrak{D}_{\alpha}^{+}(\gt)> 0 \quad \textrm{for every }\, \gt\in[g]\,.
$$
From Lemma \ref{l-intid} we have
\begin{align}\label{eq127}
\int_{M}|\nabla W^{+}|^{2}\,dV = \int_{M}\Big(4|\delta W^{+}|^{2}-\frac{1}{2}R|W^{+}|^{2}+3 W^{+}_{ijkl}W^{+}_{ijpq}W^{+}_{klpq}\Big)\,dV
\end{align}
On the other hand, since the following sharp inequality holds
\begin{equation}\label{eq999}
W^{+}_{ijkl}W^{+}_{ijpq}W^{+}_{klpq}\leq \frac{1}{\sqrt{6}}|W^{+}|^{3}\,,
\end{equation}
by H\"older inequality one has
\begin{align*}
\int_{M} W^{+}_{ijkl}W^{+}_{ijpq}W^{+}_{klpq}\,dV &\leq \frac{1}{\sqrt{6}} \int_{M} |W^{+}|^{3}\,dV \\
&\leq \frac{1}{\sqrt{6}}\left( \int_{M}|W^{+}|^{2}\,dV\right)^{\frac{1}{2}} \left( \int_{M}|W^{+}|^{4}\,dV\right)^{\frac{1}{2}}\,.
\end{align*}
Moreover, the equality case is attained if and only if $W^+$ has at most two different eigenvalues and $|W^{+}|$ is constant almost everywhere. In particular, in this case, since $W^{+}\not\equiv 0$, $|W^{+}|>0$ on $W^+$ has exactly two distinct eigenvalues on $M^{4}$. The Yamabe-Sobolev inequality applied to $u:=|W^{+}|$ yelds
\begin{align*}
\int_{M} W^{+}_{ijkl}W^{+}_{ijpq}W^{+}_{klpq}\,dV &\leq \frac{1}{\sqrt{6}\,\mathcal{Y}(M,[g])}\left( \int_{M}|W^{+}|^{2}\,dV\right)^{\frac{1}{2}}\left( 6 \int_{M}|\nabla |W^{+}||^{2}\,dV + \int_{M}R|W^{+}|^{2}\,dV \right)\\
&\leq \alpha \int_{M}|\nabla |W^{+}||^{2}\,dV + \frac{\alpha}{6} \int_{M}R|W^{+}|^{2}\,dV \,,
\end{align*} 
where in the last inequality we have used the assumption \eqref{asss}. Let 
$$
M_0:=\{x\in M: |W^+|(x)=0\}\,.
$$ 
Note that, in general, $\operatorname{Vol}(M_0)$ can be strictly positive (by a unique continuation principle, this is not the case when $\delta W^+\equiv 0$, see for instance \cite{gursky}). Then one has
\begin{align*}
\int_{M} W^{+}_{ijkl}W^{+}_{ijpq}W^{+}_{klpq}\,dV \leq \alpha \int_{M\setminus M_0}|\nabla |W^{+}||^{2}\,dV + \frac{\alpha}{6} \int_{M_0}R|W^{+}|^{2}\,dV \,.
\end{align*} 
Thus, the improved Kato inequality in Lemma \ref{l-IKI}, implies for every $k\in\left[0,\frac{5}{3}\right)$
\begin{align*}
\int_{M} W^{+}_{ijkl}W^{+}_{ijpq}W^{+}_{klpq}\,dV &\leq  \frac{\alpha}{k} \int_{M\setminus M_0}|\nabla W^{+}|^{2}\,dV + \frac{8\alpha(k-1)}{k(5-3k)} \int_{M\setminus M_0}|\delta W^{+}|^{2}\,dV \\
&\,+\frac{\alpha}{6} \int_{M\setminus M_0}R|W^{+}|^{2}\,dV \,.
\end{align*}
On the other hand, by Remark \ref{rem127}, on $M_0$ we have 
$$
|\nabla W^+|^2 \geq \frac{8}{5} |\delta W^+|^2 \,,
$$
hence
\begin{align*}
 \frac{\alpha}{k} \int_{M_0}|\nabla W^{+}|^{2}\,dV + \frac{8\alpha(k-1)}{k(5-3k)} \int_{M_0}|\delta W^{+}|^{2}\,dV+\frac{\alpha}{6} \int_{M_0}R|W^{+}|^{2}\,dV\geq 0  \,.
\end{align*}
Combining the above inequalities with \eqref{eq127}, we obtain
\begin{align*}
\frac{k-3\alpha}{k}\int_{M}|\nabla W^{+}|^{2}\,dV &\leq \frac{4k(5-3k)+24\alpha(k-1)}{k(5-3k)}\int_{M}|\delta W^{+}|^{2}-\frac{1-\alpha}{2}\int_{M}R|W^{+}|^{2}\,dV \,.
\end{align*}
Now choose $k=3\alpha$ and we get
$$
\int_{M}|\delta W^{+}|^{2}\,dV\geq\frac{5-9\alpha}{24}\int_{M}R|W^{+}|^{2}\,dV\,,
$$
i.e.
$$
\mathfrak{D}_{\alpha}^{+}(g)\geq 0\,.
$$
If $\mathfrak{D}_{\alpha}^{+}(g)> 0$, since all the assumptions are conformally invariant, this estimate holds for every metric in the conformal class $\gt\in[g]$ and the claim follows. On the other hand, suppose that $\mathfrak{D}_{\alpha}^{+}(g)= 0$. Then 
$$
\int_{M}|\delta W^{+}|^{2}\,dV=\frac{5-9\alpha}{24}\int_{M}R|W^{+}|^{2}\,dV\,.
$$
From a previous estimate, since $k=3\alpha$, we obtain
\begin{align*}
\int_{M} W^{+}_{ijkl}W^{+}_{ijpq}W^{+}_{klpq}\,dV &\leq  \frac{1}{3} \int_{M}|\nabla W^{+}|^{2}\,dV + \frac{8(3\alpha-1)}{3(5-9\alpha)} \int_{M}|\delta W^{+}|^{2}\,dV +\frac{\alpha}{6} \int_{M}R|W^{+}|^{2}\,dV\\
&=  \frac{1}{3} \int_{M}|\nabla W^{+}|^{2}\,dV + \left(\frac{3\alpha-1}{9}+\frac{\alpha}{6} \right)\int_{M}R|W^{+}|^{2}\,dV\\
&= \frac{1}{3} \int_{M}|\nabla W^{\pm}|^{2}\,dV + \frac{9\alpha-2}{18}\int_{M}R|W^{+}|^{2}\,dV\,.
\end{align*}
Thus
\begin{align*}
\int_{M}|\nabla W^{+}|^{2}\,dV \leq\left(\frac{5-9\alpha}{6}-\frac{1}{2}+\frac{9\alpha-2}{6}\right)  \int_M R|W^{+}|^{2}\,dV+\int_{M}|\nabla W^{+}|^{2}=\int_{M}|\nabla W^{+}|^{2}\,dV
\end{align*}
In particular we have equalities in the previous computations, so $|W^{+}|$ is a positive constant and the equality case in the Yamabe-Sobolev inequality gives that also the scalar curvature $R$ has to be a positive constant. Substituting in \eqref{eq127}, we obtain
\begin{align*}
\int_{M}|\nabla W^{+}|^{2}\,dV &=\left(\frac{5-9\alpha}{6}-\frac{1}{2}+\frac{\alpha}{2}\right)  \int_M R|W^{+}|^{2}\,dV=\frac{1-3\alpha}{3} \int_M R|W^{+}|^{2}\,dV\\
&=\left(\frac{1-3\alpha}{3}\right) \operatorname{Vol}(M) R |W^+|^2\,.
\end{align*}
This implies $\alpha\leq \frac13$. To conclude we use the fact that we have equality also in the Kato inequality in Lemma \ref{l-IKI} with $k=3\alpha$, i.e.
$$
|\nabla W^{+}|^{2}= \frac{8(1-3\alpha)}{(5-9\alpha)}|\delta W^{+}|^{2}
$$
on $M^4$, since $|W^+|>0$. First of all, by the equality in the algebraic estimate \eqref{eq999} we know that $W^+$ has exactly two distinct eigenvalues. Following the proof in Lemma \ref{l-IKI}, since $\operatorname{det}(W^{+})>0$, we can assume that $\mu=\lambda$ and $\nu=-2\lambda$. Thus $\bar{c}=0$ and $Z=0$. Substituting in \eqref{eq_prelmodW} and \eqref{eq_terzadivW}, we obtain
$$
\abs{\nabla W^{+}}^2 = 24|d\lambda|^{2}  +8|X|^{2}+ 8|Y|^{2}
$$
$$
\abs{\delta W^{+}}^2 = 6|d\lambda|^{2} +2|X|^{2}+ 2|Y|^{2}-6\pair{d\lambda, Y}+6\pair{d\lambda, Y}-2\pair{X, Y}\,.
$$
Thus
\begin{align*}
0=|\nabla W^{+}|^{2}- \frac{8(1-3\alpha)}{(5-9\alpha)}|\delta W^{+}|^{2} &=\frac{8}{5-9\alpha}\Big((9(1-\alpha)|d\lambda|^2+3(1-\alpha)|X|^2+3(1-\alpha)|Y|^2\\
&+6(1-3\alpha)\pair{d\lambda, X}-6(1-3\alpha)\pair{d\lambda, Y}+2(1-3\alpha)\pair{X, Y}\Big)\,.
\end{align*}
Following again the notation in Lemma \ref{l-IKI}, the associated matrix is given by
\[
\mathcal{M}=\frac{8}{5-9\alpha}\left[\begin{array}{lll}9(1-\alpha)&3(1-3\alpha)&-3(1-3\alpha)\\3(1-3\alpha)&3(1-\alpha)&(1-3\alpha)\\-3(1-3\alpha)&(1-3\alpha)&3(1-\alpha)\end{array}\right]\,.
\]
A computation shows that $\det(\mathcal{M})=288\alpha(2-3\alpha)$, which has to be zero. This is a contradiction, since $0<\alpha\leq\frac13$ and the proof of Theorem \ref{t-rigidity} is complete.

\

Corollary \ref{c-asf} follows from Theorem \ref{t-rigidity} using the lower bound for the Yamabe invariant proved in \cite{gur127} (see Lemma \ref{lem-gur}).

\

\

\begin{ackn}
\noindent The authors are members of the Gruppo Nazionale per l'Analisi Matematica, la Probabilit\`{a} e le loro Applicazioni (GNAMPA) of the Istituto Nazionale di Alta Matematica (INdAM). The first and the third authors are supported by the PRIN-2015KB9WPT Grant "Variational methods, with applications to problems in mathematical physics and geometry".
\end{ackn}

\

\

\bibliographystyle{abbrv}

\bibliography{WHW}

\

\end{document}